\documentclass{amsart}
\usepackage[utf8]{inputenc}

\usepackage{tikz}
\usetikzlibrary{cd}

\usepackage{amsmath,amssymb,amsfonts,mathrsfs,mathtools, braket}
\usepackage{enumerate}
\newtheorem{theorem}{Theorem}[section]
\newtheorem{lemma}[theorem]{Lemma}
\newtheorem{corollary}[theorem]{Corollary}

\newtheorem{definition}[theorem]{Definition}
\newtheorem{remark}[theorem]{Remark}

\numberwithin{equation}{section}

\DeclarePairedDelimiter\abs{\lvert}{\rvert}

\newcommand{\R}{\mathbb{R}}
\newcommand{\C}{\mathbb{C}}
\newcommand{\N}{\mathbb{N}}

\newcommand{\calO}{\mathcal{O}}
\newcommand{\calP}{\mathcal{P}}
\newcommand{\calQ}{\mathcal{Q}}

\newcommand{\Q}{\mathcal{Q}}

\newcommand{\id}{\mathrm{id}}

\newcommand{\GL}{\mathrm{GL}}
\newcommand{\dd}{\mathrm{d}}

\newcommand{\inv}{{-1}}
\newcommand{\invT}{{-T}}

\title{Dilational Symmetries of Decomposition and Coorbit Spaces}
\author{Hartmut F\"uhr, Reihaneh Raisi-Tousi}

\address{H.~F{\"u}hr\\
Lehrstuhl f{\"u}r  Geometrie und Analysis \\
RWTH Aachen University \\
 D-52056 Aachen\\
 Germany \\ email: fuehr@mathga.rwth-aachen.de}

 \address{R. Raisi-Tousi\\
 Department of Pure Mathematics\\ Ferdowsi University of Mashhad \\ P.O. Box 1159-91775, Mashhad \\ Iran \\ 
Hochschule Neubrandenburg \\ University of Applied Sciences \\
Institute of Applied Mathematics and Informatics in Science and Technology\\
Brodaer Stra\ss{}e 2, 17033 Neubrandenburg \\ Germany \\ 
 email: raisi@um.ac.ir; }
 
\date{\today}

\begin{document}

\begin{abstract}
    We investigate  the invariance properties of general wavelet coorbit spaces and Besov-type decomposition spaces under dilations by matrices. We show that these matrices can be characterized by quasi-isometry properties with respect to a certain metric in frequency domain. We formulate versions of this phenomenon both for the decomposition and coorbit space settings. 
    
    We then apply the general results to a particular class of dilation groups, the so-called shearlet dilation groups. We present a general, algebraic characterization of matrices the are coorbit compatible with a given shearlet dilation group. We determine the groups of compatible dilations for a variety of concrete examples. 
\end{abstract}

\maketitle

\section{Introduction}

Coorbit spaces formalize a general idea underlying wavelet approximation theory, namely that smoothness properties are closely related to decay properties of the coefficients with respect to suitable families of elementary building blocks, see \cite{FeiGroeAUniefApprToIntegrGroupRepAndTheirAtomicDec,FeiGroeBanSpaRelToIntGrpRepI,FeiGroeBanSpaRelToIntGrpRepII}. It has been understood in the meantime that this formalism can be applied to a large variety of different groups, thus allowing to efficiently capture the different approximation theoretic properties of the various systems of building blocks. In particular, generalized wavelet transforms, based on building blocks arising from a combination of translations and dilations (the latter taken from a suitably chosen matrix groups) acting on a suitably chosen single vector, provide a rich source of diverse examples. The references  \cite{DahlkeShearletCoorbitSpacesCompactlySupported,DahlkeShearletCoorbitSpacesAssociatedToFrames,FuehCooSpaAndWavCoeDecOveGenDilGro,FuRT,FePa,fischer2018heisenbergmodulation} provide a small sample of the literature; a more comprehensive list of the existing constructions can be found in the introduction of \cite{FuKo_Coarse}.

While these results illustrate the scope of the initial definitions and results of coorbit space theory, being applicable to a large variety of fundamentally different dilation groups, they also point to the necessity to develop ways to understand the similarities and difference of the various groups. For example, the problem of understanding when two different dilation groups (corresponding to two different ways of constructing a wavelet system from a single mother wavelet) have the same approximation theoretic behaviour, as coded by the associated coorbit spaces, has been clarified only fairly recently \cite{VoigtlaenderEmbeddingsOfDecompositionSpaces,FuKo_Coarse}.

An alternative approach to understanding the various coorbit spaces consists in comparing them to more classical, isotropic smoothness spaces. The recent paper \cite{FuKo_Embedding} presents a particular case of such an investigation, elucidating the embedding behaviour of general coorbit spaces into isotropic Sobolev spaces. The more recent paper \cite{FuBa} performs a similar task for anisotropic Besov spaces, which can also be viewed as (generalized) coorbit spaces \cite{fuehrvelthoven2020coorbit}. A motivating factor for undertaking these investigations was understanding the ways in which the choices entering the construction of generalized wavelet system, in particular the choice of dilation group, influence the approximation-theoretic properties of the associated wavelet systems. 

In this paper, we contribute to the understanding of wavelet coorbit spaces by studying their \textit{symmetries}, more specifically those that arise from a simple dilation $f \mapsto f \circ A$, for an invertible matrix $A$. We will derive sharp criteria for the matrices $A$ that allow to extend this dilation operator from $L^2(\mathbb{R}^d)$ to general Besov-type coorbit spaces, using coarse geometric methods. 

For a case study illustrating the applicability and scope of our general results, we will then consider the class of shearlet dilation groups in arbitrary dimensions. This class is particularly suited for such a case study: On the one hand, it is increasingly rich as the dimension $d$ grows. On the other hand, the structural features of shearlet dilation groups give rise to a rather close interplay between coorbit theory on one end, and the algebraic structures underlying the construction of shearlet dilation groups on the other. 

\subsection{Overview of the paper}

The remainder of the paper is structured as follows: Section 2 contains our general results on symmetries of decomposition and coorbit spaces. The fundamentals concerning decomposition spaces and their matrix symmetries are the subject of Subsection 2.1. As a novelty, we introduce the group $\mathcal{S}_{\mathcal{P}}$ of  \textit{decomposition space compatible matrices}, associated to an admissible covering $\mathcal{P}$, in Definition \ref{defn:dec_compatible}. This matrix group can be understood as a symmetry group of the decomposition spaces associated to the covering. 

The chief purpose of this paper is the characterization of decomposition space compatible matrices, for various classes of admissible coverings. Our main general result in this respect is Theorem \ref{thm:char_sym_group}, which provides a coarse geometric characterization, in a spirit similar to \cite{FuKo_Coarse}. Subsection 2.2 then turns to generalized wavelets and their associated coorbit spaces. We introduce the group $\mathcal{S}_{Co_H}$ of \textit{coorbit compatible matrices}, a coorbit space analog to $\mathcal{S}_{\mathcal{P}}$, in Definition \ref{defn:coorbit_compatible}. Using the decomposition space description of coorbit spaces (established in \cite{FuehrVoigtlWavCooSpaViewAsDecSpa}) allows to prove a general characterization of coorbit compatible matrices in Theorem \ref{thm:main_char_compatible_coorbit}, and clarifies that coorbit compatibility is indeed a special case of decomposition space compatibility. 

The remainder of the paper serves to illustrate the scope of the general theorems, in particular of Theorem \ref{thm:main_char_compatible_coorbit}. Sections 3 and 4 focus on generalized shearlet groups. Section 3 recalls the main ingredients that enter the construction of such groups. The algebraic structures underlying generalized shearlet dilation groups allow to translate the coarse geometric conditions from the previous section to rather stringent algebraic relations; these results are summarized in Corollary \ref{cor:summary_shearlet}. 

The final section contains explicit computations of the groups of coorbit compatible matrices for a variety of examples, namely for all admissible dilation groups in dimension 2, as well as the standard and Toeplitz shearlet dilation groups in arbitrary dimensions. The examples underscore the variety of behaviours that the coorbit spaces of different dilation groups can exhibit, e.g. as reflected as reflected in the varying dimensions. In addition, the relative ease with which these symmetry groups are computed serves to highlight the usefulness of Theorem \ref{thm:main_char_compatible_coorbit} and Corollary \ref{cor:summary_shearlet}.

\section{Decomposition and coorbit spaces: Relevant definitions and results}

In this section we discuss decomposition spaces and wavelet coorbit spaces, and the matrices that leave these spaces invariant. We first consider (and solve) the more general problem of understanding dilation invariance of decomposition spaces, before specializing on wavelet coorbit spaces. 

\subsection{Decomposition spaces}

Decomposition spaces were conceived by Feichtinger and Gr\"obner in \cite{FeichtingerGroebnerBanachSpacesOfDistributions}, and later revisited (and somewhat updated) by Borup and Nielsen \cite{BorupNielsenFrameDecompositionOfDecOfSpaces} and Voigtlaender \cite{Voigtlaender2015PHD}. In the following, we mostly rely on the last source, and focus on decomposition spaces of Besov type, associated to weighted mixed $L^p$ norms. 

The starting point for the definition of these spaces is the notion of an {\em admissible covering} $\mathcal{Q}=(Q_i)_{i\in I}$ of some open set $\mathcal{O} \subset \mathbb{R}^d$ (\cite{FeichtingerGroebnerBanachSpacesOfDistributions}), i.e. a family of nonempty sets $Q_i \subset \mathbb{R}^d$ such that
    \begin{enumerate}[i)]
        \item $\bigcup_{i\in I} Q_i = \mathcal{O}$ and
        \item $\sup_{i\in I} \sharp \{j\in I: Q_i \cap Q_j \neq \emptyset\}<\infty$.
    \end{enumerate} 
  Throughout this paper, we will concentrate on the class of  {\em (tight) structured admissible covering}, see
  Definition 2.5 of \cite{VoigtlaenderEmbeddingsOfDecompositionSpaces}. This means that
  $Q_i = T_i Q + b_i$ with $T_i\in \mathrm{GL}(\mathbb{R}^d)$, $b_i\in \mathbb{R}^d$ with an open, precompact set $Q$, and the involved matrices fulfill 
  \begin{equation} \label{eqn:str_cov_norm}
   \sup_{i,j \in I : Q_i \cap Q_j \not= \emptyset} \| T_i^{-1} T_j \| < \infty. 
  \end{equation}

    The next ingredient is a special partition of unity $\Phi=(\varphi_i)_{i\in I}$ subordinate to $\mathcal{Q}$, also called $\mathrm{L}^p$-BAPU (bounded admissible partition of unity), with the following properties
    \begin{enumerate}[i)]
        \item $\varphi_i \in C_c^\infty(\mathcal{O})\quad \forall i\in I$,
        \item $\sum_{i\in I} \varphi_i(x)=1 \quad \forall x\in \mathcal{O}$,
        \item $\varphi_i(x)=0$ for $x\in \mathbb{R}^d\setminus Q_i$ and $i\in I$,
        \item if $1\leq p \leq \infty$: $\sup_{i\in I}\Vert \mathcal{F}^{-1} \varphi_i\Vert_{\mathrm{L}^1}<\infty$,\\
                if $0<p<1$:\quad $\sup_{i\in I}|\det(T_i)|^{\frac{1}{p}-1}\Vert \mathcal{F}^{-1} \varphi_i\Vert_{\mathrm{L}^p}<\infty$.
    \end{enumerate}
    Here, $\mathcal{F}$ denotes the usual Fourier transform of a function in $\mathrm{L^2}(\R^d)$ defined by
    \begin{align*}
    \mathcal{F}f(\xi):= \int_{\mathbb{R}^d} f(x)e^{-2\pi i\langle x,\xi \rangle}\mathrm{d}x
    \end{align*}
    for $\xi \in \R^d$. We also use the notation $\widehat{f}:=\mathcal{F}(f)$.
    The final ingredient is a weight  $(u_i)_{i\in I}$ such that there exists $C>0$ with $u_i \leq C u_j$ for all $i,j \in I: Q_i \cap Q_j \neq \emptyset$. A weight with this property is also called {\em $\mathcal{Q}$-moderate}. The interpretation of this property is that the value of $(u_i)_{i\in I}$ is comparable for indices corresponding to sets which are "close" to each other. Finally, we define the \textit{(Fourier-side) decomposition space with respect to the covering $\mathcal{Q}$ and the weight $(u_i)_{i \in I}$ with integrability exponents $0<p,q\leq \infty$} as
    \begin{align}
        \mathcal{D}(\mathcal{Q}, \mathrm{L}^p, \ell^q_u):=\{f\in \mathcal{D}'(\mathcal{O}): \Vert f\Vert_{\mathcal{D}(\mathcal{Q}, \mathrm{L}^p, \ell^q_u)}< \infty\}
    \end{align}
    for 
    \begin{align}
        \Vert f\Vert_{\mathcal{D}(\mathcal{Q}, \mathrm{L}^p, \ell^q_u)}:=
        \left\Vert\left(u_i \cdot \Vert \mathcal{F}^{-1}(\varphi_i f) \Vert_{\mathrm{L}^p(\mathbb{R}^d)} \right)_{i\in I}\right\Vert_{\ell^q(I)}.
    \end{align}
    As the notation suggests, the decomposition spaces are independent of the precise choice of $\Phi$; see  \cite{Voigtlaender2015PHD} Corollary 3.4.11. 

A crucial concept for the analysis of decomposition coverings is the definition of the set of neighbors in a covering.

\begin{definition}[\cite{FeichtingerGroebnerBanachSpacesOfDistributions} Definition 2.3]\label{def:SetOfNeighbors}
For a covering $\Q = (Q_i)_{i\in I}$ of $\mathcal{O}$ with $Q_i\subset \calO$ for all $i\in I$, we define the \textit{set of neighbors of a subset $J\subset I$} as
\begin{align*}
J^*:=\Set{i\in I| \exists j\in J:\ Q_i \cap Q_j \neq \emptyset}.
\end{align*}
By induction, we set $J^{0*}:= J$ and $J^{(n+1)*} = \left(J^{n*}\right)^*$ for $n\in \N_0$. Moreover, we use the shorthand notations $i^{k*}:=\Set{i}^{k*}$ and define $Q_i^{k*}:=\bigcup_{j\in i^{k*}}Q_{j}$
for $i\in I$ and $k\in \N_0$. 
\end{definition} 

Recall that the overall aim of this paper is to clarify dilation invariance properties of various function spaces. The following definition formalizes these properties for the setting of decomposition spaces.

\begin{definition} \label{defn:dec_compatible}
Let $\mathcal{P}=(P_j)_{j\in J}$ be a structured admissible covering of the open set $\mathcal{O}\subset\R^{d}$. Given an invertible matrix $A \in {\rm GL}(d,\mathbb{R})$, we call $A$ \textbf{(decomposition space) compatible with $\mathcal{P}$} if for all $0<p,q< \infty$ and all $\mathcal{P}$-moderate weights $v$, one has 
\[
\forall f \in C_c^\infty(\mathcal{O} \cap A \mathcal{O})~:~ \| f \circ A^{-1} \|_{D(\mathcal{Q},L^{p},\ell^{q}_{v})} \asymp  \| f \|_{D(\mathcal{Q},L^{p},\ell^{q}_{v})} 
\]
We let
\[
\mathcal{S}_{\mathcal{P}} = \{ A  \in {\rm GL}(d,\mathbb{R}) : A
{\color{red}
\mbox{ is  compatible  with }} \mathcal{P} \}
\]
\end{definition}

The following observation is immediate from the definition. 
\begin{remark} \label{rem:SP_subgroup}
$\calO\subset \R^d$ be open and let $\calQ = (Q_i)_{i\in I}$ be an admissible covering. Then $\mathcal{S}_\mathcal{P} \subset {\rm GL}(d,\mathbb{R})$ is a subgroup. In particular, $A \in \mathcal{S}_{\mathcal{P}}$ iff $A^{-1} \in \mathcal{S}_{\mathcal{P}} $.

It is currently open whether $\mathcal{S}_{\mathcal{P}}$ is generally closed.
\end{remark}

The results and techniques from \cite{Voigtlaender2015PHD} suggest that an understanding of $\mathcal{S}_{\mathcal{P}}$ hinges on the comparison of different coverings, and the next definitions provide the pertinent vocabulary for such a comparison.  While our exposition follows \cite{Voigtlaender2015PHD}, most of the definitions can be traced back to \cite{FeichtingerGroebnerBanachSpacesOfDistributions}.

\begin{definition}[\cite{Voigtlaender2015PHD} Definition 3.3.1.]\label{def:NeighborsAndEquivalence}
Let $\Q = (Q_i)_{i\in I}$ and $\mathcal{P} = (P_j)_{j\in J}$ be families of subsets of $\R^d$.
\begin{enumerate}[i)]

\item We define the \textit{set of $\mathcal{P}$-neighbors of} $i\in I $ by $J_i:= \Set{j\in J | Q_i\cap P_j \neq \emptyset}.$
More generally, we call $J_i$ and $I_j$ \textit{intersection sets} for the coverings $\Q$ and $\calP$.

\item We call $\Q$ \textit{weakly subordinate} to $\mathcal{P}$ if $N(\Q, \mathcal{P}):=\sup_{i\in I}\abs{J_i} < \infty$.

The quantity $N(\mathcal{P}, \Q)$ is defined analogously, and we call $\Q$ and $\mathcal{P}$ \textit{weakly equivalent} if $N(\mathcal{P}, \Q)<\infty$ and $N(\Q, \mathcal{P})<\infty$.

\item We call $\Q$ \textit{almost subordinate} to $\mathcal{P}$ if there exists $k\in \N_0$ such that for every $i\in I$ there exists an $j_i\in J$ with $Q_i\subset P_{j_i}^{k*}$. If $k=0$ is a valid choice, then we call $\Q$ \textit{subordinate} to $\calP$.
\end{enumerate}
\end{definition}

The relevance of these notions, in particular of weak equivalence, is spelled out in the next two lemmas. Note that the formulation of the next lemma is a special case of the cited result. 

\begin{lemma}[\cite{VoigtlaenderEmbeddingsOfDecompositionSpaces}, Theorem 6.9]\label{cor:WeakEquivalenceOfInducedCoverings}
Let $\Q=(Q_i)_{i\in I}, \mathcal{P}=(P_j)_{j\in J}$ be two structured admissible coverings of the open set $\mathcal{O}\subset\R^{d}$. If $\Q$ and $\mathcal{P}$ are not weakly equivalent, then
\[
\mathcal{D}(\Q, L^{p}, \ell^{q}_{u_1})\neq \mathcal{D}(\mathcal{P}, L^{p}, \ell^{q}_{u_2})
\]
for all $\Q$-moderate weights $u_1:I\to (0,\infty)$, for all $\mathcal{P}$-moderate weights $u_2:J\to (0,\infty)$ and all $p, q \in (0,\infty]$ with $(p,q) \neq (2,2)$.
\end{lemma}

The exception $(p,q)\neq (2,2)$ is necessary to exclude trivial cases: In the case of $(p,q) = (2,2)$ the associated decomposition spaces are just weighted $\mathrm{L}^2$ spaces, by the Plancherel Theorem. In particular, if $\mathcal{O} \subset \mathbb{R}^d$ is open and of full measure, and the weight $v$ is constant, one finds that $\mathcal{D}(\mathcal{P},L^2,\ell^2_v) = L^2(\mathbb{R}^d)$, for all admissible coverings $\mathcal{P}$. 

Weak subordinateness and equivalence of coverings are important assumptions for a multitude of sufficient criteria for embeddings of decomposition spaces and their equality, as developped in \cite{VoigtlaenderEmbeddingsOfDecompositionSpaces}. The following statement is \cite[Lemma 6.11]{VoigtlaenderEmbeddingsOfDecompositionSpaces}. The statement about the range $0 \le p,q \le \infty$ is justified by the remark following the cited lemma.   
\begin{lemma} \label{lem:weak_equiv_suff}
Let $1\leq p,q\leq \infty$ and let $\emptyset\neq\mathcal{O}\subset \R^d$ be open. Further, let $\Q=(Q_i)_{i\in I}, \mathcal{P}=(P_j)_{j\in J}$ be two tight structured admissible coverings of $\mathcal{O}$, and let $u_1$ be a $\mathcal{Q}$-moderate weight and $u_2$ a $\mathcal{P}$-moderate weight.

If $\mathcal{Q}$ and $\mathcal{P}$ are weakly equivalent and there exists $C>0$ such that $C^{-1}u_1(i) \leq u_2(j)\leq Cu_1(i)$ for all $i\in I$ and $j\in J$ with $Q_i\cap P_j \neq \emptyset$, then 
\begin{align*}
\mathcal{D}(\Q, L^{p}, \ell^{q}_{u_1})= \mathcal{D}(\mathcal{P}, L^{p}, \ell^{q}_{u_2})
\end{align*}
with equivalent norms.

If all sets in the coverings are connected, the conclusion holds for the range 
$0 \le p,q \le \infty$. 
\end{lemma}

The following corollary summarizes the significance of weak equivalence for the subsequent discussion. Note that some implications hold under less stringent assumptions. Part (c) points out an important \textit{rigidity} property of decomposition spaces: Whenever two scales of decomposition spaces coincide in a single nontrivial case, they coincide everywhere. 
\begin{corollary} \label{cor:significance_weak_equivalence}
Let $\mathcal{P}= (P_j)_{j \in J},\mathcal{Q} = (Q_i)_{i \in I}$ denote two tight, structured admissible coverings of $\mathcal{O}$, each consisting of connected sets. Then the following are equivalent: 
\begin{enumerate}
    \item[(a)] $\mathcal{P}$ and $\mathcal{Q}$ are weakly equivalent. 
    \item[(b)] For all moderate weights $u_1$ on $I$ and $u_2$ on $J$ satisfying $C^{-1}u_1(i) \leq u_2(j)\leq Cu_1(i)$ whenever $Q_i\cap P_j \neq \emptyset$, with a global constant $C \ge 1$, and all $p,q \in (0,\infty]$, we have 
    \[
    \mathcal{D}(\Q, L^{p}, \ell^{q}_{u_1})= \mathcal{D}(\mathcal{P}, L^{p}, \ell^{q}_{u_2})~,
    \] with equivalent norms. 
    \item[(c)] There exist $p,q \in (0,\infty], (p,q) \not= (2,2)$ such that 
        \[
    \mathcal{D}(\Q, L^{p}, \ell^{q})= \mathcal{D}(\mathcal{P}, L^{p}, \ell^{q})~.
    \] 
\end{enumerate}

\end{corollary}
\begin{proof}
The implication (a) $\Rightarrow$ (b) follows from Lemma \ref{lem:weak_equiv_suff}, (b) $\Rightarrow$ (c) is trivial, and $(c) \Rightarrow$ (a) is due to Lemma \ref{cor:WeakEquivalenceOfInducedCoverings}. 
\end{proof}

The next result clarifies the influence that the support set $\mathcal{O}$ has on the scale of decomposition spaces.

\begin{theorem} \label{thm:different_freq_supp}
 Let $\emptyset \not= \mathcal{O},\mathcal{O}' \subset \mathbb{R}^d$ open. Let $\mathcal{Q} = (Q_i)_{i \in I}$ denote an admissible covering of $\mathcal{O}$, $\mathcal{P} = (P_j)_{j \in J}$ denote an admissible covering of $\mathcal{O}'$. 
 Assume that either $\mathcal{O}' \cap \partial \mathcal{O} \not= \emptyset$ or  $\mathcal{O} \cap \partial \mathcal{O}' \not= \emptyset$ holds, and that $\mathcal{O} \cap \mathcal{O}'$ is unbounded. Let $p_1,p_2,q_1,q_2 \in (0,\infty]$. Then 
 \[
  \forall f \in C_c^\infty(\mathcal{O} \cap \mathcal{O}')~:~ \| f \|_{D(\mathcal{Q},L^{p_1},\ell^{q_1}_{v})} \asymp  \| f \|_{D(\mathcal{P},L^{p_2},\ell^{q_2}_{w})}
 \] can only hold in the trivial case, i.e., when $(p_1,q_1) = (2,2) = (p_2,q_2)$ and $v_i \asymp w_j$ whenever $Q_i \cap P_j \not= \emptyset$. 
\end{theorem}

Note that the assumptions on $\mathcal{O},\mathcal{O}'$ are fulfilled if they are distinct open and dense subsets. Density and openness of $\mathcal{O}$ imply $\partial \mathcal{O} = \mathbb{R}^d \setminus \mathcal{O}$, and thus $\mathcal{O}' \cap \partial \mathcal{O} = \emptyset $ can only happen if $\mathcal{O}' \subsetneq \mathcal{O}$. In that case however we get $\partial \mathcal{O}' \cap \mathcal{O} = (\mathbb{R}^d \setminus \mathcal{O}') \cap \mathcal{O} \not= \emptyset$. Furthermore, $\mathcal{O} \cap \mathcal{O}'$ is dense, hence unbounded.

Our next aim is a metric characterization of weak equivalence. For this purpose we need to introduce the metric induced by a covering: 

\begin{definition}
Let $\calO\subset \R^d$ be open and let $\calQ = (Q_i)_{i\in I}$ be a covering of $\calO$. For $x,y\in \calO$, we say $x$ and $y$ are connected by a \textit{$\calQ$-chain (of length m)} if there exist $Q_1, \ldots, Q_m \in \calQ$ such that $x\in Q_1$, $y\in Q_m$ and $Q_k\cap Q_{k+1}\neq \emptyset$ for all $k\in \{1,\ldots, m-1\}$.
\end{definition}

We next define the metric, which is called \textit{$\mathcal{Q}$-chain distance} in \cite[Definition 3.4]{FeiGroeBanSpaRelToIntGrpRepI}.

\begin{definition}
Let $\calO\subset \R^d$ be open and let $\calQ = (Q_i)_{i\in I}$ be a covering of $\calO$. Define the map $d_\calQ:\calO \times \calO \to \N_0\cup \{\infty\}$ by
\begin{align*}
d_\calQ(x,y)=
\begin{cases}
\inf\Set{m\in \N |
\begin{array}{l}
x,y \text{ are connected by a}\\
\calQ\text{-chain of length } m
\end{array}
},
& x\neq y\\
0, &x=y,
\end{cases}
\end{align*}
where we set $\inf\emptyset=\infty$.
\end{definition}

\begin{remark} \label{rem:connectedness_1}
The above map in fact defines a metric on $\calO$, without any further restrictions on the covering. 
Observe that the metrics used in this paper are allowed to take the value $\infty$, just as in the precursor paper \cite{FuKo_Coarse}. If the covering $\mathcal{Q}$ consists of connected sets, the statement $d_\mathcal{Q}(x,y) < \infty$ is equivalent to the fact that $x,y$ are contained in the same connected component in $\mathcal{O}$. Ultimately this slight extension of the standard definition of metrics is innocuous in the setting we study in this paper; we refer to \cite{FuKo_Coarse}, specifically to Remark 3.20 therein, for more details. 
\end{remark}

\begin{definition}[cf. \cite{NowakYu2012} Definition 1.3.4]\label{def:quasi-isometry}
Let $(X, d_X)$ and $(Y, d_Y)$ be metric spaces. A map $f:X\to Y$ is a \textit{quasi-isometry} if the following conditions are satisfied:
\begin{enumerate}[i)]
\item The map $f$ is a \textit{quasi-isometric embedding}. This means there exist constants $L,C>0$ such that
\begin{align*}
L^{-1}d_X(x,x')-C \leq d_Y\left(f(x),f(x')\right) \leq L d_X(x,x') + C
\end{align*}
for all $x,x'\in X$.
\item The map $f$ is \textit{coarsely surjective}. This means that there exists $K>0$ such that for every $y\in Y$ exists an $x\in X$ with $d_Y\left(f(x), y\right)\leq K$.
\end{enumerate}
\end{definition}

The following result provides the metric reformulation of weak equivalence; see Theorem 3.22 of \cite{FuKo_Coarse}.
\begin{theorem}\label{thm:WeakEquivQuasiIso}
Let $\calO\subset \R^d$ be open and let $\calQ = (Q_i)_{i\in I}$ and $\calP=(P_j)_{j\in J}$ be structured admissible coverings of $\calO$ comprised of open connected subsets of $\calO$. Then the following statements are equivalent:
\begin{enumerate}[i)]
\item The coverings $\calQ$ and $\calP$ are weakly equivalent.
\item The map $\id:(\calO, d_\calQ) \to (\calO, d_\calP),\ x\mapsto x$ is a quasi-isometry.
\end{enumerate}
\end{theorem}

The following following characterization of decomposition space compatible dilations is one of the central results of this paper: 

\begin{theorem}\label{thm:char_sym_group}
Let  $\calO\subset \R^d$ be open and dense, and let $\calQ = (Q_i)_{i\in I}$ be a structured admissible covering consisting of connected subsets. Let $A \in {\rm GL}(d,\mathbb{R})$. Then $A \in \mathcal{S}_\mathcal{P}$ iff the following two conditions hold: 
\begin{enumerate}
    \item[(i)] $A \mathcal{O} = \mathcal{O}$.
    \item[(ii)] The map $\varphi_A: (\mathcal{O},d_\mathcal{P}) \to (\mathcal{O},d_{\mathcal{P}})$, $z \mapsto A z$ is a quasi-isometry. 
\end{enumerate}
\end{theorem}
\begin{proof}
We first show necessity of the conditions. Hence assume that $A \in \mathcal{S}_\mathcal{P}$. 

Let $\mathcal{O}' = A \mathcal{O}$. Define the covering $\mathcal{P} = (P_i)_{i \in I}$ of $\mathcal{O}'$, with $P_i = A Q_i$. Then it is straightforward to see that $\mathcal{P}$ is a structured admissible covering of $\mathcal{O}'$, consisting of connected subsets. In addition, if $(\varphi_i)_{i \in I}$ is a BAPU subordinate to $\mathcal{P}$, then $(\varphi \circ A^{-1})_{i \in I}$ is a BAPU subordinate to $\mathcal{Q}$. 

We next observe that 
\[
(f \circ A^{-1}) \cdot \varphi_i = (f \cdot \psi_i) \circ A^{-1}~. 
\] As a consequence,
\[
\mathcal{F}^{-1} ((f \circ A) \cdot \varphi_i) = |\det(A)|^{-1} \left( \mathcal{F}^{-1}( f \cdot \psi_i) \right) \circ A^{-T}
\] where we used the notation $A^{-T} = (A^T)^{-1}$. This entails
\[
\| \mathcal{F}^{-1} ((f \circ A) \cdot \varphi_i) \|_p = |\det(A)|^{1/p-1} \|  \mathcal{F}^{-1} (f \cdot \psi_i) \|_p~, 
\]
which immediately entails the norm equivalence 
\begin{equation} \label{eqn:norm_eq_dilated} 
\| f \circ A^{-1} \|_{\mathcal{D}(\mathcal{Q}, L^p,\ell^q)} \asymp \| f \|_{\mathcal{D}(\mathcal{P},L^p,\ell^q)}~,
\end{equation}
for all $0<p,q < \infty$. On the other hand, the assumption gives rise to 
\[
\| f \|_{\mathcal{D}(\mathcal{Q},L^p,\ell^q}) \asymp \| f \circ A^{-1} \|_{\mathcal{D}(\mathcal{Q}, L^p,\ell^q)}~,
\]
which combines with the previous norm equivalence to yield  
\[
\| f \circ A^{-1} \|_{\mathcal{D}(\mathcal{Q},L^p,\ell^q)} \asymp \| f \circ A^{-1} \|_{\mathcal{D}(\mathcal{P},L^p,\ell^q)} 
\] and finally
\begin{equation} \label{eqn:norm_equiv_decsp}
\| f \|_{\mathcal{D}(\mathcal{Q},L^p,\ell^q)} \asymp \| f \|_{\mathcal{D}(\mathcal{P},L^p,\ell^q)} ~.
\end{equation} Each of these norm equivalences holds for $f \in C_c^{\infty}(\mathcal{O}\cap \mathcal{O}')$. Since both $\mathcal{O}$ and $\mathcal{O}'$ are open and dense, this entails via Theorem \ref{thm:different_freq_supp} that $\mathcal{O} = A \mathcal{O}$, i.e. condition (i). 

In order to establish condition (ii), note that the norm equivalence (\ref{eqn:norm_equiv_decsp}) on $C_c^\infty(\mathcal{O})$ entails equality of the associated decomposition spaces, i.e.,  
\begin{equation} \label{eqn:equal_decsp}
\forall 1\le p,q < \infty~:~\mathcal{D}(\mathcal{Q},L^p,\ell^q) = \mathcal{D}(\mathcal{Q},L^p,\ell^q)~.
\end{equation}
To see this, observe that by Theorem 5.6 of \cite{Voigtlaender_atomic_decspace}, there is a countable system $(\gamma_{\kappa})_{\kappa \in K} \subset  C_c^\infty( \mathcal{O})$ such that each 
$f \in \mathcal{D}(\mathcal{Q},L^p,\ell^q)$ can be written as 
\begin{equation} \label{eqn:atomic_decspace}
f = \sum_{\kappa \in K} c_\kappa \gamma_\kappa~,
\end{equation} converging unconditionally in $\| \cdot \|_{\mathcal{D}(\mathcal{Q},L^p,\ell^q)}$ as well as in $\mathcal{D}'(O)$. By choice of the $\gamma_\kappa \in C_c^\infty(\mathcal{O})$, the norm equivalence (\ref{eqn:norm_equiv_decsp}) applies to each finite partial sum on the right hand side, resulting in the convergence of (\ref{eqn:atomic_decspace}) in $\mathcal{D}(\mathcal{P},L^p,\ell^q)$ as well. Thus we have shown $ \mathcal{D}(\mathcal{Q},L^p,\ell^q) \subset  \mathcal{D}(\mathcal{P},L^p,\ell^q)$, and the converse follows by symmetry. 

But now Corollary \ref{cor:significance_weak_equivalence} (c) $\Rightarrow$ (a) entails weak equivalence of $\mathcal{P}$ and $\mathcal{Q}$, whence \ref{thm:WeakEquivQuasiIso} allows to conclude that 
\[
{\rm id}: (\mathcal{O},d_{\mathcal{Q}}) \to (\mathcal{O},d_{\mathcal{P}}) 
\] is a quasi-isometry. On the other hand, by construction of $\mathcal{P}$ from $\mathcal{Q}$, and of the associated metrics, the map 
\[
(\mathcal{O},d_{\mathcal{P}}) \to (\mathcal{O},d_{\mathcal{Q}})~,~ \xi \mapsto A \xi~,
\]
is in fact an isometry. 
In summary, we get that 
\[
(\mathcal{O},d_{\mathcal{Q})} \to (\mathcal{O},d_{\mathcal{Q}}) ~,~ \xi \mapsto A\xi 
\] 
is the composition of a quasi-isometry and an isometry, hence a quasi-isometry, which is condition (ii).

The converse direction is obtained largely by the same arguments: Assuming (i) and (ii), we conclude that $\mathcal{P}$ is weakly equivalent to $\mathcal{Q}$ (as defined above), which in turn implies via Corollary \ref{cor:WeakEquivalenceOfInducedCoverings} the equality of the associated decomposition spaces, and the equivalence of the associated norms. Now equation \ref{eqn:norm_eq_dilated} finishes the proof. 
\end{proof}

\subsection{Admissible dilation groups}

For a closed matrix group $H\leq \GL(\R^d)$, which we also call \textit{dilation group} in the following, we define the group $G:=\R^d\rtimes H$, generated by dilations with elements of $H$ and arbitrary translations, with the group law $(x,h)\circ (y,g) := (x+hy, hg)$. We denote integration with respect to a left Haar measure on $H$ with $\dd h$, the associated left Haar measure on $G$ is then given by $d(x,h)=\abs{\det h}^{-1}  dx dh$. The Lebesgue spaces on $G$ are always defined through integration with respect to a Haar measure. The group $G$ acts on the space $\mathrm{L}^2(\R^d)$ through the \textit{quasi-regular representation} $\pi$ defined by $[\pi(x,h)f](y):=\abs{\det h}^{-1/2} f(h^\inv (y-x))$ for $f\in \mathrm{L}^2(\R^d)$. The \textit{generalized continuous wavelet transform (with respect to $\psi\in \mathrm{L^2}(\R^d)$)} of $f$ is then given as the function
$W_\psi f:G \to \C: (x,h)\mapsto \braket{f, \pi(x,h)\psi}.$ Important properties of the map $W_\psi:f\mapsto W_\psi f$ depend on $H$ and the chosen $\psi$. If the quasi-regular representation is \textit{square-integrable}, which means that there exists $\psi\neq 0$ with $W_\psi \psi \in \mathrm{L}^2(G)$, and irreducible, then we call $H$ \textit{admissible} and the map $W_\psi:\mathrm{L}^2(\R^d) \to \mathrm{L}^2(G)$ is a multiple of an isometry, which gives rise to the (weak-sense)
    inversion formula
    \begin{equation} \label{eqn:waverec}
     f = \frac{1}{C_\psi}\int_G W_\psi f(x,h) \pi(x,h) \psi \dd(x,h) ~,
    \end{equation}
i.e., each $f \in \mathrm{L}^2(\mathbb{R}^d)$  is a continuous superposition of the wavelet system.
According to results in \cite{FuehrWaveletFramesAndAdmissibilityInHigherDImensions}, \cite{FuehrGeneralizedCalderonConditionsAndRegularOrbitSpaces}, the admissibility of $H$ can be characterized by the \textit{dual action} defined by $H \times \mathbb{R}^d \to  \R^d, (h,\xi) \mapsto h^{-T} \xi$. In fact, $H$ is admissible iff the dual action has a single open orbit $\mathcal{O}:=H^\invT \xi_0\subset \R^d$ of full measure for some $\xi_0\in \R^d$ and additionally the isotropy group $H_{\xi_0}:=\Set{h :p_{\xi_0}(h)=\xi_0}\subset H$ is compact; see e.g. \cite{FuehrGeneralizedCalderonConditionsAndRegularOrbitSpaces}.

Every admissible group gives rise to an associated admissible covering. This is done using the {\em dual action} by picking a {\em well-spread} family in $H$, i.e. a family of elements $(h_i)_{i\in I}  \subset H$ with the properties
    \begin{enumerate}[i)]
        \item there exists a relatively compact neighborhood $U\subset H$ of the identity such that $\bigcup_{i\in I}h_i U= H$ -- we say $(h_i)_{i\in I}$ is \textit{$U$-dense} in this case -- and
        \item there exists a neighborhood $V\subset H$ of the identity such that $h_iV \cap h_jV = \emptyset$ for $i\neq j$ -- we say $(h_i)_{i\in I}$ is \textit{$V$-separated} in this case.
    \end{enumerate}
    The  \textit{dual covering induced by $H$} is then given by the family  $\mathcal{Q}=(Q_i)_{i\in I}$, where $Q_i = p_{\xi_0}(h_i U)$ for some $\xi_0$ with $H^{-T}\xi_0=\mathcal{O}$. It can be shown that well-spread families always exist, and that the induced covering is indeed a tight structured admissible covering in the sense defined above. In particular, $\mathrm{L}^p$-BAPUs exist for this covering, according to Theorems 4.4.6 and 4.4.13 of \cite{Voigtlaender2015PHD}. Furthermore, there always exist induced coverings consisting of open and connected sets. For ease of reference, we state this as a lemma.

\begin{lemma}[\cite{KochDoktorarbeit} Corollary 2.5.9]\label{cor:ExistenceInducedConnectedCovering} Let $H$ denote an admissible dilation group, with open dual orbit $\mathcal{O}$. Then there always exists an induced covering of $\mathcal{O}$ by $H$ that is a tight structured admissible covering consisting of (path-) connected open sets.
\end{lemma}

We call any induced covering that is a structured admissible covering consisting of open and connected sets an \textit{induced connected covering of $\mathcal{O}$ by }$H$. Note that, two different induced coverings of the same group are always weakly equivalent, see e.g. \cite{KochDoktorarbeit} Corollary 2.6.5. This fact can be understood as a consequence of the decomposition space description of wavelet coorbit spaces in Theorem \ref{thm:FourierIsoCoorbitDecSpaces} below.

\subsection{Generalized wavelet coorbit spaces}
Coorbit spaces are defined in terms of the decay behavior of the generalized wavelet transform. 
To give a precise definition, we introduce weighted mixed $\mathrm{L}^p$-spaces on $G$, denoted by $\mathrm{L}^{p,q}_v(G)$ . By definition, this space is the set of functions
\begin{align*}
        \left\{ f:G\to \mathbb{C} : \int_H\left( \int_{\mathbb{R}^d} \left| f(x,h) \right|^p v(x,h)^p \dd x \right)^{q/p}\frac{\dd h}{|\det(h)|} <\infty \right\},
    \end{align*}
    with natural (quasi-)norm $\Vert \cdot\Vert_{\mathrm{L}^{p,q}_v}$. This definition is valid for $0< p,q <\infty$, for $p=\infty$ or $q=\infty$ the essential supremum has to be taken at the appropriate place instead. The function $v:G\to \mathbb{R}^{>0}$ is a measurable weight function that fulfills the condition $v(ghk)\leq v_0(g)v(h)v_0(k)$ for some submultiplicative measurable weight $v_0$. If the last condition is satisfied, we call $v$ left- and right moderate with respect to $v_0$. Thus, the expression $\Vert W_\psi f\Vert_{\mathrm{L}^{p,q}_v}$ can be read as a measure of wavelet coefficient decay of $f$. We consider weights which only depend on $H$. The coorbit space $\mathrm{Co}\left(\mathrm{L}^{p,q}_v(\mathbb{R}^d\rtimes H)\right)$ is then defined as the space
    \begin{align}\label{def:Coorbit}
        \left\{ f\in \mathcal{(H}^1_w)^\neg : W_\psi f \in W(\mathrm{L}^{p,q}_v(G))\right\}
    \end{align}
    for a suitable wavelet $\psi$ fulfilling various technical conditions, and some control weight $w$ associated to $v$. The space $(\mathcal{H}^1_w)^\neg$ denotes the space of antilinear functionals on 
$
        \mathcal{H}^1_w :=\left\{ f\in \mathrm{L}^2(\mathbb{R}^d): W_\psi f \in \mathrm{L}^1_w(G)\right\}
$. 
    Given a Banach function space $Y$ on $G$, we let $W(Y)$ denote the Wiener amalgam space defined by
$
    W_Q(Y):=\{f\in \mathrm{L}^\infty_{\text{loc}}(G) | M_Qf\in Y\}
 $
    with quasi-norm $\|f\|_{W_Q(Y)}:=\|M_Qf\|_Y$ for $f\in  W_Q(Y)$. Here we used the \textit{maximal function} $M_Qf$ for some suitable unit neighborhood $Q\subset G$, given by $
    M_Qf:G\to [0,\infty],\ x\mapsto \operatorname*{ess\ sup}_{y\in xQ}|f(y)|.$
   
   The appearance of the Wiener amalgam space in (\ref{def:Coorbit}) is necessary to guarantee consistently defined quasi-Banach spaces in the case $\{p,q\}\cap (0,1)\neq \emptyset$, see \cite{Rauhut2007CoorbSpacTheoForQuasiBanSpa} and \cite{Voigtlaender2015PHD}. In the classical coorbit theory for Banach spaces, which was developed in \cite{FeiGroeBanSpaRelToIntGrpRepI}, \cite{FeiGroeBanSpaRelToIntGrpRepII}, the Wiener amalgam space is replaced by $\mathrm{L}^{p,q}_v(G)$ and this change leads to the same space for $p,q\geq 1$, see \cite{Rauhut2007CoorbSpacTheoForQuasiBanSpa}.
   
  Many useful properties of these spaces are known and hold in the quasi-Banach space case as well as in the Banach space case. The most prominent examples of coorbit spaces associated to generalized wavelet transforms are the homogeneous Besov spaces and the modulation spaces. However, each shearlet group, a class of groups we introduce in the next subsection, gives rise to its own scale of coorbit spaces, as well; see \cite{kutyniok2012shearlets}, \cite{DahlkeHaeuTesCooSpaTheForTheToeSheTra} and \cite{FuehCooSpaAndWavCoeDecOveGenDilGro}.
  
\begin{remark} \label{rem:coorbit_embed_L2} For $0\le p,q \le 2$ and constant weights $v$, the coorbit spaces $Co(L^{p,q}_v))$ has a canonical embedding into $L^2(\mathbb{R}^d)$; see Remark 2.15 of \cite{FuKo_Coarse}.  Here the appeal to the anti-dual  $(\mathcal{H}^1_w)^\neg$  can be avoided, i.e. one can simply define
\[
Co(L^{p,q}_v) =         \left\{ f\in L^2(\mathbb{R}^d) : W_\psi f \in W(\mathrm{L}^{p,q}_v(G))\right\}~.
\] This has the useful implication that coorbit spaces associated to different dilation groups can be compared in a straightforward manner. 
\end{remark}
 
%
     
The next definition will be useful for the transfer of weights from the coorbit to the decomposition space setting.      
     
\begin{definition}[\cite{Voigtlaender2015PHD} Definition 4.5.3.]\label{def:DecompositionWeight}
For $q\in (0,\infty]$ and a weight $m: H \to (0,\infty)$, we define the weight
$
m^{(q)}:H\to (0,\infty),\ h\mapsto \abs{\det(h)}^{\frac{1}{2} - \frac{1}{q}} m(h).
$ Here, we set $\frac{1}{\infty}:=0$. 
\end{definition}

The connection between coorbit spaces and decomposition spaces is given by the next theorem. For the Banach space case, we also refer to \cite{FuehrVoigtlWavCooSpaViewAsDecSpa}. A recent extension beyond the irreducible setting can be found in \cite{fuehrvelthoven2020coorbit}. 

\begin{theorem}[\cite{Voigtlaender2015PHD} Theorem 4.6.3]\label{thm:FourierIsoCoorbitDecSpaces}
Let $\mathcal{Q}$ be a covering of the dual orbit $\mathcal{O}$ induced by $H$, $0<p,q\leq\infty$ and $u=(u_i)_{i\in I}$ a suitable weight, then the Fourier transform
$
       \mathcal{F}: \mathrm{Co}\left(\mathrm{L}^{p,q}_v(\mathbb{R}^d\rtimes H)\right) \to \mathcal{D}(\mathcal{Q}, \mathrm{L}^p, \ell^q_u)
$
    is an isomorphism of (quasi-) Banach spaces. The weight $(u_i)_{i\in I}$ can be chosen as $u_i:=v^{(q)}(h_i)$, where $(h_i)_{i\in I}$ is the well-spread family used in the construction of $\mathcal{Q}$ and we call such a weight a $\mathcal{Q}-$discretization of $v$.
\end{theorem}

\begin{remark} \label{rem:ind_weight_intr}
 In the following, we will mostly concentrate on constant weights, i.e. on the study of coorbit spaces of the type $Co(L^{p,q}(G))$ corresponding to $v\equiv 1$. This has the important consequence that the $\mathcal{Q}$-discretization $(u_i)_{i \in I}$ obtained from a dual covering $\mathcal{Q} = (Q_i)_{i \in I} = (h_i^{-T} Q)_{i \in I}$ fulfills
 \[
  u_i = |{\rm det}(h_i)|^{\frac{1}{2}-\frac{1}{q}} = |Q|^{\frac{1}{2}-\frac{1}{q}} |Q_i|^{\frac{1}{q}-\frac{1}{2}} \asymp |Q_i|^{\frac{1}{q}-\frac{1}{2}} ~.
 \] In the parlance of \cite[Definition 2.7]{FuKo_Coarse}, the induced weight $(u_i)_{i \in I}$ is \textbf{intrinsic} with exponent $\alpha =\frac{1}{q}-\frac{1}{2}$. 
\end{remark}

\begin{remark} \label{rem:coorbit_decspace_four}
Note that the domain of the Fourier transform in the previous theorem requires some additional clarification, which can be found in Remark 2.15 of \cite{FuKo_Coarse}. For the following discussion it will be sufficient to recall the already mentioned inclusion of $Co(L^{p,q}_v) \subset L^2(\mathbb{R}^d)$ for $1 \le p,q \le 2$ and constant weight $v$. With this identification the Fourier transform from Theorem \ref{thm:FourierIsoCoorbitDecSpaces} coincides with the Plancherel transform on $L^2(\mathbb{R}^d)$.
\end{remark}

We next formalize the property that two admissible dilation groups have the same coorbit spaces. We already pointed out that a literal interpretation of this property is not generally available, at least not for all possible choices of coorbit space norms. 

\begin{definition} \label{defn:coorbit_equivalent}
 Let $H_1, H_2 \le GL(\mathbb{R}^d)$ denote admissible matrix groups. We call $H_1,H_2$ \textit{coorbit equivalent} if for all $0 < p,q \le \infty$ and for all $f \in L^2(\mathbb{R}^d)$ we have 
 \[
 \| f \|_{ Co(L^{p,q}(\mathbb{R}^d \rtimes H_1))} \asymp \| f \|_{Co(L^{p,q}(\mathbb{R}^d \rtimes H_2))}~.
 \] 
Here the norm equivalence is understood in the generalized sense that one side is infinite iff the other side is. 
\end{definition}

\begin{remark}
 An example of distinct groups that are coorbit equivalent can be found in \cite{FuehrVoigtlWavCooSpaViewAsDecSpa}, Section 9: If $H = \mathbb{R}^+ \times SO(d)$, for $d>1$, and $C \in GL(\mathbb{R}^d)$ is arbitrary, then $H$ and $C^{-1} H C$ are coorbit equivalent, but typically distinct.  
\end{remark}

The question whether two groups are coorbit equivalent can now be answered using the metric criteria for decomposition spaces. The following result is essentially \cite[Theorem 4.17]{FuKo_Coarse}. Note that condition (c) follows by combining the decomposition space description of coorbit spaces from Theorem \ref{thm:FourierIsoCoorbitDecSpaces} with the rigidity property of decomposition spaces, formulated after Lemma \ref{lem:weak_equiv_suff}.

\begin{theorem} \label{thm:coorbit_equiv_dual_orbits}
 Let $H_1, H_2 \le {\rm GL}(\mathbb{R}^d)$ denote admissible matrix groups, and let $\mathcal{O}_1,\mathcal{O}_2$ denote the associated open dual orbits. Then the following are equivalent:
 \begin{enumerate}
  \item[(a)] $H_1$ and $H_2$ are coorbit equivalent. 
  \item[(b)] For all $1 \le p,q \le 2$: $Co(L^{p,q}(\mathbb{R}^d \rtimes H_1)) = Co(L^{p,q}(\mathbb{R}^d \rtimes H_2))$, as subspaces of $L^2(\mathbb{R}^d)$. 
  \item[(c)] There exists $1 \le p,q \le 2$ with $(p,q) \not= (2,2)$, such that $Co(L^{p,q}(\mathbb{R}^d \rtimes H_1)) = Co(L^{p,q}(\mathbb{R}^d \rtimes H_2))$, as subspaces of $L^2(\mathbb{R}^d)$. 
  \item[(d)] $\mathcal{O}_1 =  \mathcal{O}_2$, and the coverings induced by $H_1$ and $H_2$ on the common open orbit are weakly equivalent. 
 \end{enumerate}
\end{theorem}

Following the cue of coorbit equivalence, we now introduce a notion of coorbit compatible matrices that focuses on invariance of certain coorbit spaces. 
\begin{definition} \label{defn:coorbit_compatible}
Let $H< {\rm GL}(\mathbb{R}^d)$ denote an admissible matrix group, and $A \in {\rm GL}(\mathbb{R}^d)$. We call $A$ \textbf{coorbit compatible with $H$}, if for all  $0 < p,q \le \infty$ and for all $f \in L^2(\mathbb{R}^d)$ we have 
 \[
 \| f \|_{ Co(L^{p,q}(\mathbb{R}^d \rtimes H))} \asymp \| f \circ A^{-1} \|_{Co(L^{p,q}(\mathbb{R}^d \rtimes H))}~.
 \] 
 We let 
 \[ 
 \mathcal{S}_{Co_H}  = \{ A \in {\rm GL}(\mathbb{R}^d) : A \mbox{ is coorbit compatible with } H \} \]
\end{definition}

Just as for $\mathcal{S}_{\mathcal{P}}$, we immediately obtain that $\mathcal{S}_{Co_H}$ is a group: 
\begin{remark} \label{rem:SCo_subgroup}
If $H< {\rm GL}(\mathbb{R}^d)$ is an arbitrary admissible matrix group, then $\mathcal{S}_{Co_H} \subset {\rm GL}(\mathbb{R}^d)$ is a subgroup. In particular, $A \in \mathcal{S}_{Co_H}$ iff $A^{-1} \in \mathcal{S}_{Co_H} $. Furthermore, given $A \in H$, we can compute  
\[
W_\psi (f \circ A) (x,h) = (W_\psi f)((0,A^{-1})\cdot (x,h))~,
\] and now left invariance of the spaces $L^{p,q}_v(\mathbb{R}^d \rtimes H)$ entails that $A \in \mathcal{S}_{Co_H}$, showing the inclusion $H \subset \mathcal{S}_{Co_H}$. It is currently not known whether $\mathcal{S}_{Co_H}$ is generally closed.
\end{remark}

 We next introduce the word metrics needed to formulate criteria for coorbit compatibility at the group level. 
 
\begin{definition}
Let $H$ be a locally compact group and let $W\subset H$ be a unit neighborhood. Define the map
$d_W : H\times H \to \N_0\cup \{\infty\}$
in the following way
\begin{align*}
d_W(x,y)=
\begin{cases}
\inf\Set{m \in \N| x^\inv y \in W^m } & x\neq y\\
0& x=y,
\end{cases}
\end{align*}
where we again set $\inf\emptyset = \infty$.
\end{definition}

The following results rest on a somewhat subtle technical condition on the dual stabilizers
\[
H_\xi = \{ h \in H : h^{-T} \xi = \xi \}~,\xi \in \mathcal{O}~.
\] We will use $H_0 \subset H$ to denote the connected component of the identity element in $H$. The fact that these notations clash for the zero element $0$ is immaterial for the following, since $0 \not\in \mathcal{O}$, hence the stabilizer of $0$ does not enter the discussion. 

Throughout the following, the condition $H_\xi \subset H_0$ will repeatedly occur, where $\xi$ is an arbitrary element of $\mathcal{O}$. This condition is independent of the choice of $\xi \in \mathcal{O}$. Further observations, and a more detailed discussion of the role of this condition in the context of coorbit equivalence can be found in Section 4 of \cite{FuKo_Coarse}.

We then have the following result, which is Theorem 4.16 from \cite{FuKo_Coarse}. 
\begin{theorem}\label{thm:OrbitMapQuasiIsometry}
Assume $H_\xi \subset H_0$. Let $W \subset H$ be a relatively compact, symmetric unit neighborhood with $W\subset H_0$. Furthermore, let $\calQ = (h_i^\invT Q)_{i\in I}$ be an induced connected covering of $\calO$ by $H$  with $\xi \in Q$,  for some open, relatively compact $Q \subset \mathcal{O}$. Then 
$$p_\xi:(H, d_W) \to (\calO, d_\calQ), h\mapsto h^\invT \xi$$
is a quasi-isometry. 
\end{theorem}
 
With these observations in place, we can now formulate and prove the following characterization of compatible dilations. The theorem can be viewed as a rather natural analogue of Theorem \ref{thm:coorbit_equiv_dual_orbits}, and it is arguably the main result of our paper.

Before we formulate the theorem, we recall the definition of a \textit{word metric} on a locally compact group: Given an open symmetric neighborhood of the identity $W$ in such a group $G$, let $W^0 = \{ e_G \}$, and $W^n = \{ x_1 \cdot \ldots \cdot x_n : x_1,\ldots, x_n \in W$ for all $n \in \mathbb{N}$. Finally, define $x \not= y$
\[
d_W(x,y) = \min \{ n \in \mathbb{N} : x^{-1} y \in W^n \}~.
\] where by convention, $d_W(x,y) = \infty$ if $x^{-1} y \not\in \langle W \rangle$, the subgroup generated by $W$. In particular, if $G_0 < G$ denotes the connected component the identity element of $G$ and $W \subset G_0$ is an open symmetric neighborhood of the identity element, then $\langle W \rangle = G_0$, and $d_W(x,y) < \infty$ is equivalent to saying that $x$ and $y$ are in the same connected component of $G$. 

\begin{theorem} \label{thm:main_char_compatible_coorbit}
Let $H$ denote an admissible matrix group with open dual orbit $\mathcal{O}$. Assume that $H_\xi \subset H_0$ holds. Let $A \in {\rm GL}(d,\mathbb{R})$. Then the following are equivalent:
\begin{enumerate}
    \item[(a)] $A \in \mathcal{S}_{Co_H}$.
    \item[(b)] $H$ is coorbit equivalent to $AHA^{-1}$
    \item[(c)] For any covering $\mathcal{P}$ of $\mathcal{O}$ induced by $H$, $A^{-T} \in \mathcal{S}_\mathcal{P}$.
    \item[(d)] Let $W \subset H_0$ denote any open, symmetric unit neighborhood, and $d_W$ denote the associated word metric on $H$. Let $(p_{\xi})^{-1}: \mathcal{O} \to H$ denote an arbitrary right inverse to $p_\xi$. Then  $A^{-T} \mathcal{O} = \mathcal{O}$, and the map 
    \[
    \varphi_A : H \to H~,~ \varphi_A : h \mapsto (p_{\xi})^{-1} (A^{-T} h^{-T} .\xi) 
    \] is a quasi-isometry with respect to $d_W$.
\end{enumerate}
\end{theorem}

\begin{proof}
The equivalence (a) $\Leftrightarrow$ (b) follows from \cite[Lemma 44]{FuehrVoigtlWavCooSpaViewAsDecSpa}, combined with the canonical embedding $Co_H(L^{p,q}(\mathbb{R}^d \rtimes H)) \subset L^2(\mathbb{R}^d)$.

\underline{Proof of (a) $\Rightarrow$ (c):} Pick $g \in C_c^\infty(\mathcal{O} \cap A^{-T} \mathcal{O})$. Then $\mathcal{F}^{-1}(g) \in L^2(\mathbb{R}^d)$ with $\mathcal{F}^{-1}(g \circ A^{-T}) = |\det(A)| \mathcal{F}^{-1}(g) \circ A$. Combining Theorem \ref{thm:FourierIsoCoorbitDecSpaces}  with the assumption $A \in \mathcal{S}_{Co_H}$ now implies the following chain of norm equivalences, for any covering $\mathcal{P}$ induced by $H$
\begin{eqnarray*}
\| g \circ A^{-T} \|_{\mathcal{D}(\mathcal{P},L^p,\ell^q)} & \asymp & \| \mathcal{F}^{-1} (g \circ A^{-T}) \|_{Co(L^{p,q}(\mathbb{R}^d \rtimes H))} \\
& \asymp &  \| \mathcal{F}^{-1} (g ) \|_{Co(L^{p,q}(\mathbb{R}^d \rtimes H))}  \\
& \asymp & \| g \|_{\mathcal{D}(\mathcal{P},L^p,\ell^q)}
\end{eqnarray*}
which shows that $A^{-T} \in \mathcal{S}_{\mathcal{P}}$.

\underline{Proof of (c) $\Rightarrow$ (a):} Asssume that $A^{-T} \in \mathcal{S}_\mathcal{P}$, for a covering $\mathcal{P}$ induced by $H$. Consider $1 \le p,q \le 2$. Then we get, for all $g \in \mathcal{F}^{-1}(C_c^\infty(\mathcal{O}))$, the chain of norm equivalences
\begin{eqnarray*}
\| g \circ A \|_{Co(L^{p,q}(\mathbb{R}^d \rtimes H))} & \asymp &  \| \mathcal{F}(g \circ A) \|_{\mathcal{D}(\mathcal{P},L^p,\ell^q)} \\ 
& \asymp &  \| \mathcal{F}(g) \circ A^{-T} \|_{\mathcal{D}(\mathcal{P},L^p,\ell^q)}  \\ 
& \asymp &  \| \mathcal{F}(g)  \|_{\mathcal{D}(\mathcal{P},L^p,\ell^q)} \\
& \asymp & \| g \|_{Co(L^{p,q}(\mathbb{R}^d \rtimes H))}~.
\end{eqnarray*}
Here the second-to-last equivalence was due to $A \in \mathcal{S}_{\mathcal{P}}$, and the remaining ones due to  Theorem \ref{thm:FourierIsoCoorbitDecSpaces}. 
Hence it remains to prove that this norm equivalence holds for all $g \in L^2(\mathbb{R}^d)$ (in the extended sense), which is achieved by using a density argument similar to the one used in the proof of Theorem \ref{thm:char_sym_group}. More precisely, picking any nonzero $\psi \in \mathcal{F}^{-1}(C_c^\infty(\mathcal{O}))$, Lemma 2.7 of \cite{FuehCooSpaAndWavCoeDecOveGenDilGro} implies that $\psi$ is a so-called \textbf{better vector} in the sense of \cite{FeiGroeBanSpaRelToIntGrpRepI}. By Theorem 6.1 of the cited paper, there exists a discrete family $((x_\kappa,h_\kappa))_{\kappa \in K} \subset \mathbb{R}^d \rtimes H$ such that the system $(\pi(x_\kappa,h_\kappa) \psi)_{\kappa \in K}$ is a Banach frame both for $L^2(\mathbb{R}^d)$ and for $Co(L^{p,q}(\mathbb{R}^d \rtimes H))$. 

This means that every $f \in L^2(\mathbb{R}^d)$ can be written as
\begin{equation} \label{eqn:atomic_coorbit}
f = \sum_{\kappa \in K} c_\kappa \pi(x_\kappa, h_\kappa) \psi~,
\end{equation} with unconditional convergence in $\| \cdot \|_2$~, but also in $\| \cdot \|_{Co(L^{p,q}(\mathbb{R}^d \rtimes H))}$, as soon as $f$ is contained in the latter, smaller space. By choice of $\psi$, the finite partial sums of (\ref{eqn:atomic_coorbit}) all lie in $\mathcal{F}^{-1}(C_c^\infty(\mathcal{O}))$. 

On the other hand, $A^{-T} \in \mathcal{S}_{\mathcal{P}}$ also entails $A^{-T} \mathcal{O} = \mathcal{O}$, and the finite partial sums of the expansion
\begin{equation} \label{eqn:atomic_coorbit_dil}
 f \circ A = \sum_{\kappa \in K} c_\kappa \pi(x_\kappa, h_\kappa) \psi \circ A 
\end{equation}
therefore also lie in $\mathcal{F}^{-1}(C_c^\infty(\mathcal{O})$. Hence if $f \in Co(L^{p,q}(\mathbb{R}^d \rtimes H))$, we obtain by the already established norm equivalence that (\ref{eqn:atomic_coorbit_dil}) converges also in $\| \cdot \|_{Co(L^{p,q}(\mathbb{R}^d \rtimes H))}$, finally leading to $f \circ A \in Co(L^{p,q}(\mathbb{R}^d \rtimes H))$. In addition, the norm equivalence valid for the partial sums extends (with identical constants) to the limit $f \circ A$. 

The converse inclusion follows by symmetry, observing that by Remark \ref{rem:SCo_subgroup} the argument can be applied with $A$ systematically replaced by $A^{-1}$. 

\underline{Proof of (c) $\Leftrightarrow$ (d)}: $A^{-T} \in \mathcal{S}_{\mathcal{P}}$ implies $A^{-T} \mathcal{O} = \mathcal{O}$, as well as the quasi-isometry property of
\[
\mathcal{O} \to \mathcal{O}~,~\zeta \mapsto A^{-T} \zeta~.
\] Suitably composing this map with the quasi-isometries $p_{\xi}$ and $(p_\xi)^{-1}$  (by \ref{thm:OrbitMapQuasiIsometry}) gives the map $\varphi_A$.

The converse direction is proved analogously. 
\end{proof}

The next corollary notes a simple application of $(a) \Longleftrightarrow (b)$. It can be seen as a natural extension of the inclusion $H \subset \mathcal{S}_{Co_H}$ from Remark \ref{rem:SCo_subgroup}.
\begin{corollary} \label{cor:normalizer}
Let $H<GL(\mathbb{R}^d)$ denote an admissible subgroup, and let  
\[
N_H = N_H(GL(\mathbb{R}^d)) = \{ A \in GL(\mathbb{R}^d): AHA^{-1} = H \}~,
\] the \textit{normalizer subgroup of $H$ in $GL(\mathbb{R}^d)$}. Then $N_H \subset \mathcal{S}_{Co_H}$. In particular, $\mathbb{R}^\ast \cdot I_d \subset \mathcal{S}_{Co_H}$.
\end{corollary}


\section{A review of shearlet dilation groups}

In this section, we review the main structural features of shearlet dilation groups. In full generality,this class was first introduced and studied in \cite{FuRT}. A more recent overview with additional results can be found in \cite{AlbertiEtAl2017}.

We start out by presenting the two main examples of shearlet dilation groups, known prior to \cite{FuRT}
\begin{definition}[\cite{AlbertiEtAl2017}\label{def:StandardToeplitzGroups} Example 17. and Example 18.]\label{def:StandardToeplitzShearletGroups}
\begin{enumerate}[i)]
\item  For $\lambda=(\lambda_1, \ldots, \lambda_{d-1})\in \R^{d-1}$, we define the \textit{standard shearlet group in $d$-dimensions} $H^\lambda$ as the set
\begin{align*}
\Set{\epsilon\, \mathrm{diag}\left(a, a^{\lambda_1},\ldots, a^{\lambda_{d-1}}\right)
\begin{pmatrix}
1 & s_1 & \ldots & s_{d-1} \\
 & 1 & 0\ldots & 0			\\
 &  & \ddots & 	0	\\
 &  & & 1
\end{pmatrix}
|
\begin{array}{l}
a>0, \\
s_i\in \R,\\
\epsilon\in \Set{\pm 1}
\end{array}
}.
\end{align*}
\item For  $\delta\in \R$, we define the \textit{Toeplitz shearlet group in $d$-dimensions} $H^\delta$ as the set
\begin{align*}
\Set{\epsilon\,  \mathrm{diag}\left(a, a^{1-\delta},\ldots, a^{1-(d-1)\delta}\right)\cdot T(1, s_1, \ldots, s_{d-1})
|
\begin{array}{l}
a>0,\\
s_i\in \R,\\
\epsilon\in \Set{\pm 1}
\end{array}
},
\end{align*}
where the matrix $T(1, s_1, \ldots, s_{d-1})$ is defined by
\begin{align*}
T(1, s_1, \ldots, s_{d-1}):=
\begin{pmatrix}
1 & s_1 & s_2 & \ldots & \ldots & s_{d-1} \\
  & 1	& s_1 & s_2 & \ldots & s_{d-2}\\
  && \ddots & \ddots & \ddots & \\
  &&& 1 & s_1 & s_2\\
  &&&& 1 & s_1\\
  &&&&& 1
\end{pmatrix}.
\end{align*}
\end{enumerate}
\end{definition}

The presentation of more general shearlet dilation groups requires additional elementary notations from the area of (matrix) Lie groups. In the following, $\mathfrak{gl}(\R^d)$ denotes the set of all real $d \times d$-matrices. We let 
$$\exp: \mathfrak{gl}(\R^d) \to \mathrm{GL}(\R^d)$$
be the exponential map defined by
\[\exp(A) := \sum_{k=0}^{\infty}\frac{A^k}{k!}~.\]
Furthermore, we denote with $T(\R^d) \subset \mathrm{GL}(\R^d)$ the group of upper triangular $d\times d$-matrices with one on their diagonals. By definition, the Lie algebra of a closed subgroup $H\subset \mathrm{GL}(\R^d)$ is the set $\mathfrak{h}$ of all matrices $Y$ in $\mathfrak{gl}(\R^d)$ such that $\exp(tY)\in H$ for all $t\in \R$. 

The following definition was first formulated in \cite{FuRT}; see also \cite{AlbertiEtAl2017}.

\begin{definition}\label{defn:shearlet_dil}
Let $H\subset \mathrm{GL}(\R^d)$ be a closed, admissible dilation group. The group $H$ is called \textit{generalized shearlet dilation group} if there exist two closed subgroups 
$$S, D \subset \mathrm{GL}(\R^d)$$
such that
\begin{enumerate}[i)]
\item $S$ is a connected abelian subgroup of $T(\R^d)$;
\item $D=\Set{\exp(rY) | r\in \R}$ is a one-parameter group, where $Y\in \mathfrak{gl}(\R^d)$ is a diagonal matrix;
\item every $h\in H$ has a unique representation as $h=\pm ds$ for some $d\in D$ and $s\in S$.
\end{enumerate}
$S$ is called the \textit{shearing subgroup of $H$}\index{shearing subgroup}, $D$ is called the \textit{scaling subgroup of $H$}, and $Y$ is called the \textit{infinitesimal generator of $D$}.
\end{definition}

The article \cite{AlbertiEtAl2017} provides a systematic construction process for shearlet groups in arbitrary dimension that works by first selecting a shearing group $S$ and then determining conditions on the infinitesimal generator $Y$ of the scaling group $D$.

We denote the canonical basis of $\R^d$ with $e_1,\ldots, e_d$ and the identity matrix in $\mathrm{GL}(\R^d)$ with $I_d$. The next result contains information about the structure of shearing subgroups.

\begin{lemma}[\cite{AlbertiEtAl2017} Lemma 5. and Lemma 6.]\label{lem:CanonicalBasis}
Let $S$ be the shearing subgroup of a generalized shearlet dilation group $H\subset \mathrm{GL}(\R^d)$. Then the following statements hold:
\begin{enumerate}[i)]
\item There exists a unique basis $X_2,\ldots, X_d$ of the Lie algebra $\mathfrak{s}$ of $S$ with $X_i^Te_1=e_i$ for $2\leq i \leq d$, called the canonical basis of $\mathfrak{s}$.
\item We have $S=\Set{I_d + X | X\in \mathfrak{s}}$.
\item Let $\mathfrak{s}_k = {\rm span}\{X_j : j \ge k \}$, for $k \in \{ 2, \ldots, d \}$.  These are associative matrix algebras satisfying $\mathfrak{s}_k  \mathfrak{s}_\ell \subset \mathfrak{s}_{k+\ell-1}$, where we write  $\mathfrak{s}_m = \{ 0 \}$ for $m >d$.
\item $H$ is the inner semidirect product of the normal subgroup $S$ with the closed subgroup $D \cup -D$. 
\end{enumerate} 
\end{lemma}
%
%

Every generalized shearlet dilation group $H$ is admissible, and the next result shows that all of them share the same dual orbit. This ensures that one of the basic conditions for coorbit equivalence is already fulfilled. The following result was also first obtained in \cite{FuRT}.

\begin{lemma}\label{lem:OrbitShearletGroups}
Let $S\subset T(\R^d)$ be a shearing subgroup and $D$ be a compatible scaling subgroup such that 
$$H= DS\cup (-DS)$$
is a generalized shearlet dilation group. Then the unique open dual orbit of $H$ is given by $\mathcal{O}=\R^* \times \R^{d-1}$, and the dual action of $H$ on $\mathcal{O}$ is free.
\end{lemma}

\section{Characterizing coorbit compatible dilations for general shearlet dilation groups}

\label{sect:SCo_shearlet}
\subsection{The symmetries of the dual orbit}

Throughout this section, we let $\mathcal{O} = \mathbb{R}^* \times \mathbb{R}^{d-1}$, the dual orbit of a shearlet dilation group in dimension $d-1$. We let 
\[ \mathcal{S}(\mathcal{O}) = \{ A \in {\rm GL}(d,\mathbb{R}) : A^{-T} \mathcal{O} = \mathcal{O} \} ~,\]
then Theorem \ref{thm:main_char_compatible_coorbit} implies 
\[
\mathcal{S}_{Co_H} \subset \mathcal{S}(\mathcal{O})~.
\]
The next lemma characterizes the matrices in $\mathcal{S}(\mathcal{O})$ and notes some basic formulae that will be helpful for computations.

\begin{lemma} \label{lem:struct_SO}
Let $\mathcal{O} = \mathbb{R}^* \times \mathbb{R}^{d-1}$, and $A \in {\rm GL}(d,\mathbb{R})$. 
\begin{enumerate}
    \item[(a)] $A \in \mathcal{S}(\mathcal{O})$ holds iff
    \[
     A = \left( \begin{array}{cc} \lambda & z \\ \mathbf{0} & B \end{array} \right) 
    \] with $\lambda \in \mathbb{R}^*$, $z \in \mathbb{R}^{1\times (d-1)}$ and $B \in {\rm GL}(d-1,\mathbb{R})$. 
    \item[(b)] Given two matrices $A, A' \in \mathcal{S}(\mathcal{O})$, i.e., 
    \[
    A = \left( \begin{array}{cc} \lambda & z \\ \mathbf{0} & B \end{array} \right) ~,~ A' = \left( \begin{array}{cc} \lambda' & z' \\ \mathbf{0} & B' \end{array} \right) 
    \] then 
    \begin{equation} \label{eqn:prod_SO}
     AA' = \left( \begin{array}{cc} \lambda  \lambda' & zB' + \lambda z' \\ \mathbf{0} & BB' \end{array} \right) 
    \end{equation} and 
    \begin{equation} \label{eqn:inv_SO}
     A^{-1} = \left( \begin{array}{cc} \lambda^{-1} & -\lambda^{-1} z B^{-1} \\ \mathbf{0} & B^{-1} \end{array} \right) 
    \end{equation}
\end{enumerate}
\end{lemma}
\begin{proof}
For part (a), we first observe that the subgroup property of $\mathcal{S}(\mathcal{O})$ implies that $A \in \mathcal{S}(\mathcal{O})$ is equivalent to $A^T \mathcal{O} = \mathcal{O}$. With $\xi_0 = (1,0,\ldots,0)^T$, we note that $\xi \in \mathcal{O} $ iff $\langle \xi,\xi_0 \rangle \not= 0$. Hence we get for $A \in GL(\mathbb{R}^d)$ the following equivalences: 
\begin{eqnarray*}
A \in \mathcal{S}(\mathcal{O})  & \Longleftrightarrow & A^T \mathcal{O} \subset \mathcal{O} \\
   & \Longleftrightarrow & \forall \xi \in \mathbb{R}^d ~:~ \left(\langle \xi,\xi_0 \rangle \not= 0  \Rightarrow
   \langle A^T \xi, \xi_0 \rangle \not= 0 \right) \\
    & \Longleftrightarrow & \forall \xi \in \mathbb{R}^d ~:~ \left(\langle \xi,\xi_0 \rangle \not= 0  \Rightarrow
   \langle  \xi, A \xi_0 \rangle \not= 0 \right) \\
    & \Longleftrightarrow & A \xi_0 = \lambda \xi_0~, \lambda \not= 0 ~. 
\end{eqnarray*}
Here $\lambda$ is the scalar from part (a), and the existence of $z,B$ as in (a) follows immediately. 
Equation (\ref{eqn:prod_SO}) is a special instance of block matrix calculus, and can be employed to solve the equation $AA' = I_d$ to obtain $A' = A^{-1}$ in equation (\ref{eqn:inv_SO}). 
\end{proof}

The next lemma reduces the general problem of characterizing $\mathcal{S}_{Co_H}$ within $\mathcal{S}(\mathcal{O})$ to the characterization within a smaller subgroup. The subgroup in question is defined as 
  \[
    \mathcal{S}_1(\mathcal{O}) = \left\{ \left( \begin{array}{cc} 1 & \mathbf{0} \\ \mathbf{0} & B \end{array} \right)  : B \in {\rm GL}(d-1,\mathbb{R}) \right\}~.
\]

\begin{lemma} \label{lem:red_to_S1}
Let $A \in \mathcal{S}(\mathcal{O})$. Let $S \subset {\rm GL}(d,\mathbb{R})$ denote the shearing subgroup of a shearlet dilation group $H$. Then $A$ factors uniquely as
\[
A = \lambda \cdot h \cdot A_1~,\lambda \in \mathbb{R}^*~,~h \in S~,~A_1 \in \mathcal{S}_1(\mathcal{O}).  
\] One has the equivalence
\[
A \in \mathcal{S}_{Co_H} \Longleftrightarrow A_1 \in \mathcal{S}_{Co_H} ~. 
\]
\end{lemma}
\begin{proof}
We can clearly write $A = \lambda \cdot A'$ uniquely, where 
\[ A' = \left( \begin{array}{cc} 1 & z' \\ \mathbf{0} & B' \end{array} \right)~.  \]
Define $z'' = - z' (B')^{-1}$. Since the orbit map $p_{\xi_0}: h \to h^{-T} \xi_0 \in \mathcal{O}$ is onto,  there exists $h \in H$ with $h^{-T} \xi_0 = (1,z'')^T$, which is equivalent to 
\[
 h^{-1} =  \left( \begin{array}{cc} 1 & z'' \\ \mathbf{0} & B'' \end{array} \right)~.
\] for a suitable invertible matrix $B''$. Note that the entry $1$ in the upper left corner entails $h \in S$. Now the choice of $z''$ and part (b) of Lemma \ref{lem:struct_SO} entail that 
\[
 h^{-1} A' =  \left( \begin{array}{cc} 1 & 0 \\ \mathbf{0} &  B'' B' \end{array} \right) = A_1 \in \mathcal{S}_1(\mathcal{O}),
\] and we have shown the desired factorization 
\[
A = \lambda \cdot h \cdot A_1~.
\] By Corollary \ref{cor:normalizer} and Remark \ref{rem:SCo_subgroup} respectively, the first two factors are contained in $\mathcal{S}_{Co_H}$. Hence the product is in $\mathcal{S}_{Co_H}$ iff $A_1$ is.  
\end{proof}

\begin{remark}
\label{rem:diagonal_compatible_dils}
Let $S$ denote a shearing subgroup $S$, and let  $Y = {\rm diag}(1,\lambda_2,\ldots,\lambda_d)$, and $D = \exp(\mathbb{R} Y)$.  Following \cite{AlbertiEtAl2017}, we call $Y$ \textbf{compatible with $S$} iff $H = DS \cup -DS$ is a generalized shearlet dilation group. We stress that this notion of compatibility  is distinct from \textit{coorbit} compatibility, as formulated in Definition \ref{defn:coorbit_compatible}. In fact, as observed in \cite{AlbertiEtAl2017}, compatible matrices are characterized by the condition that $\exp(\mathbb{R} Y)$ normalizes $S$. This can be slightly rewritten as follows: If we define ${\rm Diag}(d) < {\rm GL}(\mathbb{R}^d)$ as the diagonal subgroup of $GL(\mathbb{R}^d)$, we obtain that   
\[
Y \mbox{ compatible with } S \Leftrightarrow \exp(\mathbb{R} Y) \subset N_S \cap {\rm Diag}(d)~. 
\] The benefit of this observation is the following: Given two dilation subgroups $D,D'$ that are both compatible with the same shearing subgroup $S$, then $D' \subset N_H \subset \mathcal{S}_{Co_H}$, where $H = DS \cup -DS$. To see this, observe that $D,D'$ commute elementwise, since they both consist of diagonal matrices, hence the elements of $D'$ normalize $H$. 

In other words, elements of \textit{dilation groups compatible with $S$} are \textit{coorbit compatible} with every shearlet dilation group having $S$ as shearing subgroup.  As a consequence, if $H,H'$ are different shearlet dilation groups sharing the same shearing subgroup, then $H' \subset \mathcal{S}_{Co_H}$.
\end{remark}

\subsection{Characterizing compatible dilations within $\mathcal{S}_1(\mathcal{O})$}

In order to enable concrete computations, we now introduce coordinates to the shearlet dilation group $H$. 
Recall that we have $H = DS \cup (-DS)$. The associated Lie algebra is $\mathfrak{h} = \mathbb{R} \cdot Y \oplus \mathfrak{s}$, with $Y$ the infinitesimal generator of the scaling subgroup $D$, and $\mathfrak{s}$ the Lie algebra of the shearing subgroup $S$. By Proposition 7 of \cite{AlbertiEtAl2017}, we can normalize $Y$ so that 
\[
 Y = {\rm diag}(1,\lambda_2,\ldots,\lambda_d)~.
\]
We use $\xi_0 = (1,0,\ldots,0) \in \mathcal{O}$. We parameterize group elements of $H$ as follows: Given $r \in \mathbb{R}$ and $t = (t_2,\ldots,t_d) \in \mathbb{R}^{d-1}$, we let 
\begin{eqnarray} \label{eqn:def_h}
 h(r,t) & = & \exp(-rY) \left(I_d+ \sum_{j=2}^d t_j X_j\right)^{-1} \in H~.~
\end{eqnarray} Note that $D = \{ h(r,0) : r \in \mathbb{R} \}$ and $ S = \{ h(0,t) : t \in \mathbb{R}^{d-1} \}$. 
Note further that this only parameterizes $DS$, and not all of $H$, but this will be sufficient for the coming arguments. 

We can then compute 
\[
 p_{\xi_0}^H (h(r,t)) = \exp(rY) (I_d+ \sum_{j=2}^d t_j X_j^T) \left( \begin{array}{c} 1 \\ 0 \\ \vdots \\ 0 \end{array} \right) =  \left( \begin{array}{c} e^r \\ e^{\lambda_2r} t_2 \\ \vdots \\ e^{\lambda_d r} t_d \end{array} \right) ~.
\]

The next lemma exhibits an explicit formula for the crucial map $\varphi_A : H \to H$, for $A \in \mathcal{S}_1(\mathcal{O})$, with respect to the newly introduced coordinates. 
\begin{lemma} \label{lem:compute_phiA_concr}
Let 
\[
A =  \left( \begin{array}{cc} 1 & \mathbf{0} \\ \mathbf{0} & B \end{array} \right) \in \mathcal{S}_1(\mathcal{O})~, 
\] and 
\[
\varphi_A : H \to H~,~ \varphi_A (h) = (p_{\xi_0})^{-1} (A^{-T} h^{-T} \xi_0)~. \]
Then one has 
\[
\varphi_A (h(r,t)) = h(r,\tilde{Y}_{-r} B^{-T} \tilde{Y}_r t)~, 
\]
where we used $\tilde{Y}_s = \exp(s {\rm diag}(\lambda_2,\ldots,\lambda_d))$.

In particular, the restriction $\varphi_A|_S : S \to S$ is a well-defined bijection. 
\end{lemma}

\begin{proof} Simple calculations show  that 
\begin{eqnarray*}
\varphi_A (h(r,t))  & = & (p_{\xi_0})^{-1} (A^{-T} p_{\xi_0} (h(r,t)))\\
& = &  (p_{\xi_0})^{-1} \left( \left( \begin{array}{cc} 1 & \mathbf{0} \\ \mathbf{0} & B^{-T} \end{array} \right)
  \left( \begin{array}{c} e^r \\ e^{\lambda_2r} t_2 \\ \vdots \\ e^{\lambda_d r} t_d \end{array} \right) \right) \\
& = &(p_{\xi_0})^{-1} \left(
  \left( \begin{array}{c} e^r  \\ B^{-T} \tilde{Y}_r t \end{array} \right) \right)\\&=& h(r,\tilde{Y}_{-r} B^{-T} \tilde{Y}_r t ).
\end{eqnarray*}
The statement about $\varphi_A|_S$ is then clear. 
\end{proof}

Our characterization of the elements $A \in \mathcal{S}_{Co_H} \cap \mathcal{S}_1(\mathcal{O})$ rests on two algebraic conditions on $A$. The following lemma clarifies one of these conditions: 
\begin{lemma} \label{lem:char_commutation}
Let  
\[
A =  \left( \begin{array}{cc} 1 & \mathbf{0} \\ \mathbf{0} & B \end{array} \right) \in \mathcal{S}_1(\mathcal{O})~, 
\] and $\varphi_A: H \to H$ defined as in the previous lemma.  Then the following are equivalent: 
\begin{enumerate}
    \item[(a)] $B$ commutes with $\tilde{Y} = {\rm diag}(\lambda_2,\ldots,\lambda_d)$.
    \item[(b)] $\varphi_A|_S: S \to S$ commutes with the conjugation action of $D$ on $S$, i.e.  \[
    \forall d \in D \forall s \in S~:~ d^{-1} \varphi_A(s) d = \varphi_A( d^{-1} s d)~.
    \]
    \item[(c)] For all $h(r,t) \in DS$, one has
    \[
    \varphi_A(h(r,t)) = h(r,B^{-T}t)= h(r,0) \varphi_A(h(0,t))~. 
    \]
    \item[(d)] $B^{-T}$ commutes with all matrices
    \[
    \tilde{Y}_r = \exp(r~ {\rm diag}(\lambda_2,\ldots,\lambda_d))~. 
    \]
\end{enumerate}
\end{lemma}
\begin{proof}
For $(a) \Leftrightarrow (d)$ note that $B$ commutes with $\tilde{Y}$ iff $B^{-1}$ commutes with $\tilde{Y}$, iff $B^{-T}$ commutes with $\tilde{Y}^T = \tilde{Y}$. It is straightforward to check that the last condition is equivalent to $(d)$.

$(d) \Leftrightarrow (c)$ follows directly from the formula 
\[
\varphi_A(h(r,t)) = h(r,\tilde{Y}_{-r} B^{-T} \tilde{Y}_r t) 
\] established in Lemma \ref{lem:compute_phiA_concr}.

In order to show $(d) \Leftrightarrow (b)$, we consider $d = h(r,0)$ and $s = h(0,t)$, and define $\tilde{Z}_r= {\rm diag} (e^{-r(\lambda_2 -  1} ), ...,  e^{-r(\lambda_d-1)})$. Then we have the relation
\begin{equation} \label{eqn:conj_coords}
d^{-1} h(0, t) d= h (0,   \tilde{Z}_rt).\end{equation}
To see this, we first compute 
\begin{eqnarray*}
d^{-1} h(0, t) ^{-1}d  & = & 
\begin{pmatrix}
e^{-r}      \\
 &   e^{-r\lambda _2} & 			\\
 &  & \ddots & 		\\
 &  & & e^{-r\lambda _d}
\end{pmatrix}
\begin{pmatrix}
1 &  &  t   &    \\
 &   1    &  0			\\
 &      & \ddots & 		\\
 &  & & 1 
\end{pmatrix}
\begin{pmatrix}
e^{r}      \\
 &   e^{r\lambda _2} & 			\\
 &  & \ddots & 		\\
 &  & & e^{r\lambda _d}
\end{pmatrix}
\\
& = & 
\begin{pmatrix}
1      &   e^{r(\lambda _2-1) t } &  e^{r(\lambda _3-1)t  }  &\dots &		e^{r(\lambda _d-1) t }	\\
& 1 & \\
&  &  1 & \\
 &   &   & \ddots &  	\\
  &  &   & & 1
\end{pmatrix}
\\
& = &
h(0, \tilde{Z}_r t)^{-1}  \in S.
\end{eqnarray*}
Now inverting both sides yields equation (\ref{eqn:conj_coords}). 
With the help of this equation, condition (b) is reformulated as 
\[
\forall r \in \mathbb{R}~\forall t \in \mathbb{R}^{d-1}:~ h(0,\tilde{Z}_r B^{-T} t) = h(0,B^{-T} \tilde{Z}_r t)~,
\] which is equivalent to the condition that $B^{-T}$ commutes with all $\tilde{Z}_r$. But this last condition is equivalent to (d), because of $\tilde{Z}_r = e^r \tilde{Y}_{-r}$. 
\end{proof}

\begin{remark} \label{rem:char_commutant_Y}
Since $Y$ is diagonal, the matrices commuting with $\tilde{Y}$ are particularly easy to determine: Note that $\mathbb{R}^{d-1}$ is the direct sum of eigenspaces of $\tilde{Y}$. Then it is straightforward to check that a matrix $B$ commutes with $\tilde{Y}$ iff $B$ maps each eigenspace into itself. This condition translates directly to a block diagonal structure of $B$, up to possible permutations of the entries.

More precisely, let $B = (b_{i,j})_{2 \le i,j \le d}$ be any matrix in $GL(\mathbb{R}^{d-1})$. Here we adjusted the index set for the entries to comply with the indexing conventions for $t = (t_2,\ldots,t_d) \in \mathbb{R}^{d-1}$ adopted above. 
Then $B \tilde{Y} = \tilde{Y} B$ is equivalent to requiring $b_{i,j} = 0$ whenever $\lambda_i \not= \lambda_j$.
\end{remark}

The following lemma notes an algebraic condition relating certain automorphisms of $S$ to those of $\mathfrak{s}$, and exhibits a close connection to the normalizer of $S$. The following lemma makes use of the various ways in which the shearing group $S$ can be written, using Lemma \ref{lem:CanonicalBasis} and the parametrization $h$, i.e. 
\[
 S = \{ I_d + X : X \in \mathfrak{s} \} = \left\{ I_d + \sum_{i=2}^d t_i X_i : t \in \mathbb{R}^{d-1} \right\} = \left\{ h(0,t) : t \in \mathbb{R}^{d-1} \right\}~.
\]
\begin{lemma} \label{lem:group_vs_algebra}
Let 
\[
A =  \left( \begin{array}{cc} 1 & \mathbf{0} \\ \mathbf{0} & B \end{array} \right) \in \mathcal{S}_1(\mathcal{O})~, 
\] 
with $B \in GL(\mathbb{R}^{d-1})$.
Then the following are equivalent: 
\begin{enumerate}
    \item[(a)] The map $\psi_B: h(0,t) \mapsto h(0,Bt)$ is a group automorphism of $S$. 
    \item[(b)] The linear isomorphism $\Psi_B: \mathfrak{s} \to \mathfrak{s}$, 
    \[
    \sum_{i=2}^d t_i X_i \mapsto \sum_{i=2}^d s_i X_i~,~ s= Bt
    \]
    is an automorphism of the associative matrix algebra $\mathfrak{s}$.
    \item[(c)] $A \in N_S$.
\end{enumerate}
In particular, every algebra automorphism $\Psi : \mathfrak{s} \to \mathfrak{s}$ gives rise to a unique matrix $B$ such that $\Psi = \Psi_B$, and vice versa. 
\end{lemma}
\begin{proof}
We start out by noting that $h(0,t) = (I_d + \sum_{i=2}^d t_i X_i)^{-1}$, and therefore $\psi_B(I_d + X)^{-1} = I_d + \Psi_B(X)$, for all $X \in \mathfrak{s}$, by definition of $\Psi_B$. Furthermore, since $S$ is commutative, $\psi_B$ is a group isomorphism iff $\tilde{\psi}_B$, defined by $\tilde{\psi} (h(0,t)) = \psi_B(h(0,t))^{-1}$ is an isomorphism. Using these observations, one gets via linearity of $\Psi_B$ that
\begin{eqnarray*} \lefteqn{\tilde{\psi}_B((I_d+X)(I_d+Y)) = \tilde{\psi}_B (I_d+X) \tilde{\psi}_B(I_d+Y)} \\ & \Leftrightarrow & I_d+\Psi_B(X+Y+XY) = I_d+\Psi_B(X) + \Psi_B(Y) + \Psi_B(X) \Psi_B(Y) \\
& \Leftrightarrow & \Psi_B(XY) = \Psi_B(X) \Psi_B(Y)~.  \end{eqnarray*}
In other words, the multiplicativity properties of $\tilde{\psi}_B$ on $S$ and of $\Psi_B$ on $\mathfrak{s}$ are equivalent. Since $\Psi_B$ is by definition a linear bijection, the desired equivalence follows.

The proof of the equivalence $(a) \Leftrightarrow (c)$ requires additional notation. We write 
\[
h(0,t)^{-1} = I_d + \sum_{i=2}^d t_i X_i = \left( \begin{array}{cc} 1 & t^T \\ \mathbf{0} & I_{d-1}+C(t) \end{array} \right)~,
\] with a linear map $C: \mathbb{R}^{d-1} \to \mathbb{R}^{(d-1) \times (d-1)}$ induced by the entries of the canonical basis $X_2,\ldots,X_d$ of $\mathfrak{s}$. With this notation we get the product formula
\begin{equation}
    h(0,t_1)^{-1} h(0,t_2)^{-1} = h(0,t_1+t_2 + C(t_2)^T \cdot t_1)^{-1}~.
\end{equation}
Hence, we see that the condition (a) is equivalent to the equation 
\[
h(0,B^{T} \cdot (t_1 +t_2 + C(t_2)^T \cdot t_1)) = h(0,B^T\cdot t_1 + B^T \cdot(t_2 + C(B^{T}\cdot t_2)^T \cdot B^T \cdot t_1))~,
\] holding for all $t_1,t_2 \in \mathbb{R}^{d-1}$. Using the fact that $h$ is bijective, we can simplify this condition to 
\[
\forall t \in \mathbb{R}^{d-1} ~:~ B^T \cdot C(t)^T = C(B^{T} \cdot t)\cdot B^T ~. 
\]
A further slight simplification therefore establishes 
\begin{equation} \label{eqn:char_psi_B_hom}
    \psi_B \mbox{ is a group homomorphism} \Leftrightarrow 
    \forall t \in \mathbb{R}^{d-1} ~:~ B^{-1} \cdot C(t) \cdot B = C(B^T \cdot t)~.
\end{equation}
On the other hand, a direct computation yields
\begin{eqnarray*}
A^{-1} \cdot h(0,t)^{-1} \cdot A & = &  \left( \begin{array}{cc} 1 & \mathbf{0} \\ \mathbf{0} & B \end{array} \right)  \cdot 
\left( \begin{array}{cc} 1 & t^T \\ \mathbf{0} & I_{d-1}+C(t) \end{array} \right)
\cdot  \left( \begin{array}{cc} 1 & \mathbf{0} \\ \mathbf{0} & B^{-1} \end{array} \right) \\
& = & \left( \begin{array}{cc} 1 & (B^T \cdot t)^T \\ \mathbf{0} & I_{d-1}+B^{-1} \cdot C(t) \cdot B \end{array} \right)~.
\end{eqnarray*} 
This shows that 
\[ A^{-1} \cdot h(0,t)^{-1} \cdot A \in S \] is equivalent to 
\[ A^{-1} \cdot h(0,t)^{-1} \cdot A = h(0,B^T \cdot t)^{-1} ~,\] for all $t \in \mathbb{R}^{d-1}$. Now a comparison of the lower right block matrices on both sides of the equation reveals that this is in turn equivalent to the right-hand side of the equivalence (\ref{eqn:char_psi_B_hom}). This concludes the proof of (a) $\Leftrightarrow$ (c). 

Finally, given any linear map $\Psi: \mathfrak{s} \to \mathfrak{s}$, the existence of a matrix $B$ fulfilling  
\[
\forall t \in \mathbb{R}^{d-1}~:~\Psi\left(\sum_{i=2}^d t_i X_i\right) = \sum_{i=2}^d s_i X_i~,~s = Bt
\] is a basic fact of linear algebra. The equivalence (a) $\Leftrightarrow$ (b) therefore establishes the final statement of the lemma.  
\end{proof}

We can now prove the main result of this section, which is an algebraic characterization of the elements of $\mathcal{S}_{Co_H}$.
\begin{theorem} \label{thm:charact_compat_S1}
Let $H = DS \cup -DS$ denote a shearlet dilation group, with infinitesimal generator $Y = {\rm diag}(1,\lambda_2,\ldots,\lambda_d)$ of $S$. Let \[
A =  \left( \begin{array}{cc} 1 & \mathbf{0} \\ \mathbf{0} & B \end{array} \right) \in \mathcal{S}_1(\mathcal{O})~.  
\] Then the following are equivalent: 
\begin{enumerate}
    \item[(a)] $A  \in \mathcal{S}_{Co_H}$.
    \item[(b)] $A \in N_S$, and $A$ commutes with $Y$.
    \item[(c)] $A \in N_H$.
\end{enumerate}
\end{theorem}

\begin{proof} 
The implication (b) $\Rightarrow$ (c) is proved by straightforward computation, whereas (c) $\Rightarrow$ (a) was observed in Corollary \ref{cor:normalizer}. The following argument therefore focuses on (a) $\Rightarrow$ (b).

Assume  $A \in \mathcal{S}_{Co_H}$. We aim to use Lemma \ref{lem:group_vs_algebra}, and therefore want to establish that  the map $
    \psi_B=\varphi_A|_S: S \to S$ is a group homomorphism.

For this purpose, fix $s\in \mathbb{R}^{d-1}$. Define 
\[
    \alpha_S:  \mathbb{R}^{d-1}  \to H~,~ 
    \alpha_S (t)= \varphi_A(h(0,t)h(0,s))^{-1} \varphi_A(h(0,t)), 
    t \in \mathbb{R}^{d-1}.   
    \]
Then $\alpha_S$ is a composition of  polynomial maps, hence it is polynomial (note that inversion is polynomial on the set of unipotent matrices).
Since $A \in \mathcal{S}_{Co_H}$, by Theorem \ref{thm:main_char_compatible_coorbit}, $\varphi_A$ is quasi-isometry  with respect to a suitable word metric $d_H$ on $H$. So there exist $a>0, b\geq 0$ such that 
\begin{eqnarray*}
d_H(  \alpha_S (t) , e_H) & = & d_H (\varphi_A(h(0,t)h(0,s))^{-1} \varphi_A(h(0,t)), e_H)\\
& = &  d_H (\varphi_A(h(0,t)), \varphi_A(h(0,t)h(0,s)))
\\&\leq & a d_H (h(0,t), h(0,t)h(0,s)) + b
\\
& = &  a d_H(e_H, h(0,s)) + b,
\end{eqnarray*}
using left-invariance  of the metric in the second equality.
Thus  
$\alpha_S$ is a bounded polynomial function, hence constant.
In particular, \[
\varphi_A(h(0,t)h(0,s))^{-1} \varphi_A(h(0,t))=  \alpha_S (t) = \alpha_S (0)= 
    \varphi_A(h(0,s))^{-1}.\]
    Thus 
\[
\varphi_A(h(0,t)h(0,s))= \varphi_A(h(0,t))\varphi_A(h(0,s)),\]
which implies that 
$\psi_B=\varphi_A|_S$ is a group homomorphism. Now Lemma \ref{lem:group_vs_algebra} yields $A \in N_S$.

To prove the second condition,  note that 
\[
\varphi_A (h(r,t)) = h(r,\tilde{Y}_{-r} B^{-T} \tilde{Y}_r t)~=
h(r, 0) h(0, \tilde{Y}_{-r} B^{-T} \tilde{Y}_r t),
\] with $ \tilde{Y}_r = \exp(r~ {\rm diag}(\lambda_2,\ldots,\lambda_d))$.
Fix $d=h(r,0)$. Define
\[
\beta(t)  =   \varphi_A(h(0,t)d)^{-1}\varphi_A(h(0,t))
=\varphi_A(dd^{-1}h(0,t)d)^{-1}\varphi_A(h(0,t)).
\]


Now we have:
\begin{eqnarray*}
\beta(t)  &=&  \varphi_A(dd^{-1}h(0,t)d)^{-1}\varphi_A(h(0,t))\\
& = & 
\varphi_A(h(r, 0)h(0, \tilde{Z}_r t))^{-1} \varphi_A(h(0,t))
\\&= &  [h(r, 0)h(0, \tilde{Y}_{-r} \tilde{Z}_r B^{-1} \tilde{Y}_r t)]^{-1}  h(0,B^{-T}t)\\
&=&   h(0, \tilde{Y}_{-r} \tilde{Z}_r B^{-1} \tilde{Y}_r t)^{-1}  h(-r, 0) h(0, B^{-T}t),
\end{eqnarray*}
with $r\in \mathbb{R}$ fixed, 
thus $\beta (t)$ defines   a polynomial in $t$. Furthermore, it is bounded,  since, for any word metric $d_H$ on $H$, we can employ the quasi-isometry property of $\varphi_H$ to estimate 
\begin{eqnarray*} d_H(\beta(t),e_H) &  = & d_H((\varphi(h(0,t) d)^{-1} \varphi_A(h(0,t)), e_H) \\ &  = & d(\varphi_A(h(0,t)), \varphi_A(h(0,t) d) ) \\
& \le & a d(h(0,t),h(0,t)d) + b \\
& \le & a d(e_H, d) + b \end{eqnarray*}. 

It follows that $\beta$ must be constant, i.e. 
$\beta (t) = h(-r, 0) = d^{-1} $, for all $t\in \mathbb{R}^{d-1}$.
Using  equation \ref{eqn:conj_coords}, we get
\begin{eqnarray*}
h(-r, 0) h(0, B^{-T} t )   &=& h(-r, 0) h(0, B^{-T} t )
h(r, 0) h(-r, 0)\\
& = & 
h(0, \tilde{Z}_r B^{-T} t ) h(-r, 0). 
\end{eqnarray*}
Here we recall the diagonal matrices $\tilde{Z}_r =  {\rm diag} (e^{-r(\lambda_2 -  1)} , ...,  e^{-r(\lambda_d-1)})$ from the proof of Lemma \ref{lem:char_commutation}.
In summary, the equation 
$\beta (t)= d^{-1}$ holds 
if and only if 
$h(0, \tilde{Y}_{-r} \tilde{Z}_r B^{-T} \tilde{Y}_r t)=h(0, \tilde{Z}_r B^{-T} t )$, 
for all $t\in \mathbb{R}^{d-1}$,
if 
and only if 
 $\tilde{Y}_{-r} B^{-T} \tilde{Y}_r = B^{-T} $,
 (note that $\tilde {Z}_r, \tilde{Y}_r$ commute).
 Now Lemma \ref{lem:char_commutation} (d) $\Rightarrow$ (a) yields that $B$ and $\tilde{Y}$ commute, and it follows that $A$ and $Y$ commute. 
\end{proof}

Combining Theorem \ref{thm:charact_compat_S1} with Lemma \ref{lem:red_to_S1} gives rise to the following characterization of $\mathcal{S}_{Co_H}$. Observe that the condition $(a)$, together with the final statement of Lemma \ref{lem:group_vs_algebra}, essentially reduces the task of computing $\mathcal{S}_{Co_H}$ to that of determining a subgroup of algebra automorphisms of $\mathfrak{s}$ commuting with the conjugation action of $\exp(\mathbb{R} Y)$ on $\mathfrak{s}$. 

\begin{corollary} \label{cor:summary_shearlet}
Let $A \in \mathcal{S}(\mathcal{O})$. Then $A \in \mathcal{S}_{Co_H}$ holds iff 
\[
A = \lambda \cdot h \cdot A_1~,~\lambda \in \mathbb{R}^*~,~h \in S~,~A_1  =  \left( \begin{array}{cc} 1 & \mathbf{0} \\ \mathbf{0} & B \end{array} \right) \in \mathcal{S}_1(\mathcal{O}),
\] and $A_1$ fulfills the following equivalent conditions: 
\begin{itemize}
    \item[(a)] $A_1 \in N_S$ and $A_1$ commutes with $Y$.
    \item[(b)] $A_1 \in N_H$. 
\end{itemize}
In particular, $\mathcal{S}_{Co_H} = N_H$.
\end{corollary}

As a further application, we note
\begin{corollary} \label{cor:SCo_Lie}
Let $H$ denote a shearlet dilation group. Then $\mathcal{S}_{Co_H} \subset GL(\mathbb{R}^d)$ is a closed matrix group, with
\[
d \le \dim(\mathcal{S}_{Co_H}) \le d^2-d+1~.
\]
\end{corollary}

\begin{proof} 
Since the factorization is unique, it is not hard to show that the factorization map $(\lambda,h,A_1) \mapsto \lambda h A_1$ is a diffeomorphism onto a closed subgroup. Hence the dimensions for the choices $\lambda, h,A_1$ add up to yield the dimension of $\mathcal{S}_{Co_H}$. The dimension for the first two variables are $1,d-1$, respectively, whereas 
the dimension of the set of choices for the matrix $B$ lies between $0$ and $(d-1)^2$. 
\end{proof}
 
Section \ref{sect:exs} will exhibit examples of shearlet dilation groups showing that the upper bound for the dimension of $\mathcal{S}_{Co_H}$ is sharp, as well examples $H$ with dimension ${\rm dim}(\mathcal{S}_{Co_H}) = d+1$. It is currently open whether there exist shearlet dilation groups $H$ with ${\rm dim}(\mathcal{S}_{Co_H}) = d$.

We finally use the corollary to quickly derive a correction to Theorem 5.9 \cite{FuKo_Coarse}. 
\begin{theorem}
    Let $H_1,H_2$ denote two shearlet dilation groups of equal dimension. Then $H_1$ and $H_2$ are coorbit equivalent iff $H_1 = H_2$.
\end{theorem}

\begin{proof}
    Assume that $H_1$ and $H_2$ are coorbit equivalent shearlet dilation groups. Assume that $H_i = D_i S_i \cup - D_i S_i$, with scaling subgroups $D_i$ and shearing subgroups $S_i$. While the precise formulation of \cite{FuKo_Coarse}[Theorem 5.9] is incorrect, its proof of the \textit{necessary conditions} for coorbit equivalence of shearlet dilation groups, in particular Lemmas 5.11 and 5.12 of \cite{FuKo_Coarse}, are correct. This entails $D_1 = D_2$ by \cite[Lemma 5.11]{FuKo_Coarse}. In addition, \cite[Lemma 5.12]{FuKo_Coarse} provides a matrix $C \in {\rm GL}(\mathbb{R}^d)$ such that $S_2 = C^{-1}S_1 C$, and in addition, the conjugation actions of $C$ commutes with the conjugation action of $D_1=D_2$. As a consequence, $H_2 = C^{-1} H_1 C$.

    But now the fact that $H_1$ and $H_2$ are coorbit equivalent yields via Corollary \ref{cor:summary_shearlet} that $C \in \mathcal{S}_{Co_{H_1}} = N_H$. But this means $H_2 = H_1$.
\end{proof}

\section{Determining $\mathcal{S}_{Co_H}$: Examples}

\label{sect:exs}
\subsection*{The full picture in two dimensions}

For dimension two, the admissible dilation groups have been classified up to conjugacy and finite index subgroups in \cite{FuDiss}; the following is a complete list of representatives, with their open dual orbits: 
\begin{itemize}
 \item {\bf Diagonal group} 
 \[
  D = \left\{ \left( \begin{array}{cc} a & 0 \\ 0  & b \end{array} \right) : ab \not=  0 \right\} ~,
 \]
 with $\mathcal{O} = (\mathbb{R^*})^2$.
\item {\bf Similitude group} 
 \[
  H = \left\{ \left( \begin{array}{cc} a & b \\ -b  & a \end{array} \right) : a^2+b^2 > 0 \right\} 
 \] with $\mathcal{O} = \mathbb{R}^2 \setminus \{ 0 \}$. 
\item {\bf Shearlet group(s)} For a fixed parameter $c \in \mathbb{R}$,
 \[
  S_c = \left\{ \pm \left( \begin{array}{cc} a & b \\ 0  & a^c \end{array} \right) : a > 0 \right\} ~,
 \] with $\mathcal{O} = \mathbb{R}^* \times \mathbb{R}$. 
\end{itemize}
No two distinct groups from this list are coorbit equivalent \cite{FuKo_Coarse}. We will now determine the compatible dilations for each: 
\begin{itemize}
    \item In the case of the diagonal group $D$, the requirement $A^T \mathcal{O} = \mathcal{O}$ already leads to severe restrictions. In fact, one readily sees that 
    \[
    \mathcal{S}(\mathcal{O}) = \left\{ R^{\epsilon} h~:~ h \in D, \epsilon \in \{0, 1 \} \right\}
    \] with the reflection matrix $R = \left( \begin{array}{cc} 0 & 1 \\ 1 & 0 \end{array} \right)$. 
    Since $R$ clearly normalizes $D$, Remark \ref{rem:SCo_subgroup} and Corollary \ref{cor:normalizer} entail 
    \[
    \mathcal{S}_{Co_H} = \mathcal{S}(\mathcal{O}) = N(H)~. 
    \]
    \item The compatible dilations for the similitude group have already been determined in \cite{FuehrVoigtlWavCooSpaViewAsDecSpa}, namely as 
    \[
    \mathcal{S}_{Co_H} = \mathcal{S}(\mathcal{O}) =  GL(\mathbb{R}^d)~.
    \] Note that in this case, $N(H) \subsetneq \mathcal{S}_{Co_H}$, and that unlike the diagonal case, the group of compatible dilations has strictly bigger dimension than $H$ itself. 
    \item Using the results from Section \ref{sect:SCo_shearlet}, in particular Corollary \ref{cor:summary_shearlet} we get
    \[
    \mathcal{S}_{Co_{S_c}} =  \left\{ \left( \begin{array}{cc} a & b \\ 0  & d \end{array} \right) : ad \not= 0 \right\} ~.
    \]  Again we have 
     \[
    \mathcal{S}_{Co_{S_c}} = \mathcal{S}(\mathcal{O}) = N(S_c)~. 
    \] Note that $\mathcal{S}_{Co_{S_c}}$ is independent of $c$. Since we know by \cite{FuKo_Coarse} that different choices of $c$ lead to distinct scales of coorbit space, we have found an instance where distinct scales of coorbit spaces can have the same symmetry groups. 
\end{itemize}

\subsection*{Standard and Toeplitz shearlet dilation groups in arbitrary dimensions}

The shearing subgroup of the shearlet dilation group has the Lie algebra 
\[
\mathfrak{s} = \left\{ \left( \begin{array}{cccc} 0 & t_2 & \ldots & t_d \\ 0 & \ldots & \ldots & 0 \\
\vdots & \vdots & \vdots & \vdots \\ 0 & \ldots & \ldots & 0 \end{array} \right) : t_2,\ldots,t_d \in \mathcal{R} \right\}~. 
\]
Recall from Corollary \ref{cor:summary_shearlet}, that the main challenge lies in characterizing the elements in the intersection $\mathcal{S}_1(\mathcal{O}) \cap \mathcal{S}_{Co_H}$, i.e. elements of the type \[
A = \left( \begin{array}{cc} 1 & \mathbf{0} \\ \mathbf{0} & B \end{array} \right) 
\]
and that this characterization rests on two conditions:
\begin{itemize}
\item The map $\Psi:\mathfrak{s} \ni \sum_{i=2}^d t_i X_i \mapsto \sum_{i=2}^d s_i X_i, ~s = B^{-T} t$ is an algebra automorphism of $\mathfrak{s}$. Clearly, we can replace $B$ by $B^{-1}$, and will systematically do so in the subsequent arguments. 
\item If $Y = {\rm diag}(1,\lambda_2,\ldots,\lambda_d)$ is the infinitesimal generator of the scaling subgroup, and $\tilde{Y} = {\rm diag}(\lambda_2,\ldots,\lambda_d)$, then $\tilde{Y}$ and $B$ commute.  
\end{itemize}
Here it is relevant to note that the first condition is fulfilled for every choice of $B$: The mapping $\psi$ is linear and bijective by construction. Furthermore, the associative algebra structure of $\mathfrak{s}$ is trivial in the sense that given any two matrices $a_1,a_2 \in \mathfrak{s}$, the product is given by $a_1 a_2 = 0$. Hence any \textit{linear} automorphism $\mathfrak{s} \to \mathfrak{s}$ is automatically an \textit{algebra} automorphism. Therefore the first condition is trivially fulfilled, and we get from Corollary \ref{cor:summary_shearlet} that 
\[
A = \mathcal{S}_{Co_H} \Leftrightarrow A = \lambda \cdot h \cdot \left( \begin{array}{cc} 1 & \mathbf{0} \\ \mathbf{0} & B \end{array} \right) ~,~B \in GL(\mathbb{R}^{d-1})~,~B\tilde{Y}  = \tilde{Y} B  ~.
\]

For the characterization of the second condition, we refer to Remark \ref{rem:char_commutant_Y}. The dimension of $\mathcal{S}_{Co_H}$ is now determined in terms of the eigenvalue multiplicities: If $n_1,\ldots,n_k$ are the multiplicities of the distinct eigenvalues of $\tilde{Y}$, we obtain
\[
\dim(\mathcal{S}_{Co_H}) = d+\sum_{\ell=1}^k n_{\ell}^2~.
\]
Note that any partition of $d-1$, i.e. every tuple $(n_\ell)_{\ell=1,\ldots,k}$ of positive integers satisfying
\[
\sum_{\ell=1}^k n_\ell = d-1
\]
can occur as multiplicities of a suitable choice of $Y$. The extreme cases are given by $\tilde{Y} = \lambda I_{d-1}$, which leads to
\[
 \dim(\mathcal{S}_{Co_H}) = d^2-d+1~,
\] and the multiplicity-free case, resulting in 
    \[ \dim(\mathcal{S}_{Co_H}) = 2d-1~. \]

\subsection*{Coorbit compatible dilations for Toeplitz shearlet dilation groups}

In order to describte the coorbit compatible dilations for the Toeplitz shearlet dilation groups, the main datum required to apply Corollary \ref{cor:summary_shearlet} is the automorphism group of the Lie algebra $\mathfrak{s}$ of the shearing subgroup. 

Recall that the shearing subgroup for the Toeplitz shearlet group consists of the elements
\begin{align*}
T(1, s_1, \ldots, s_{d-1}):=
\begin{pmatrix}
1 & s_1 & s_2 & \ldots & \ldots & s_{d-1} \\
  & 1	& s_1 & s_2 & \ldots & s_{d-2}\\
  && \ddots & \ddots & \ddots & \\
  &&& 1 & s_1 & s_2\\
  &&&& 1 & s_1\\
  &&&&& 1
\end{pmatrix},
\end{align*}
with $s_1,\ldots,s_{d-1} \in \mathbb{R}^{d-1}$. The associated Lie algebra then consists of the matrices

\begin{align*}
T(0, s_1, \ldots, s_{d-1}):=
\begin{pmatrix}
0 & s_1 & s_2 & \ldots & \ldots & s_{d-1} \\
  & 0	& s_1 & s_2 & \ldots & s_{d-2}\\
  && \ddots & \ddots & \ddots & \\
  &&& 0 & s_1 & s_2\\
  &&&& 0 & s_1\\
  &&&&& 0
\end{pmatrix}. 
\end{align*}
The canonical basis of $\mathfrak{s}$ (in the sense of Lemma \ref{lem:CanonicalBasis}) is given by 
\[
X_2 = T(0,1,0,\ldots,0)~,~X_3 = T(0,0,1,0,\ldots,0)~,~\ldots~,~X_d = T(0,\ldots,0,1)~,
\] and it is not hard to verify the relations
\[
\forall j=2,\ldots,d~:~ X_j = X_2^{j-1}~.
\] We proceed as in the standard shearlet case, and first determine the relevant algebra automorphisms $\Psi$ of $\mathfrak{s}$. Clearly a necessary condition for such maps is that $\Psi$ maps the generating element $X_2$ onto another generating element. Furthermore, a general element 
\[
b_2 = \sum_{j=2}^d c_j X_j~,~c_2,\ldots,c_d \in \mathbb{R}
\] is easily seen to be generating iff $c_2 \not= 0$. In such a case, letting
\[
b_j = b_2^{j-1}
\] defines a second basis $b_2,\ldots,b_d$ of $\mathfrak{s}$, and the unique linear map 
\[
\Psi: \mathfrak{s} \to \mathfrak{s}~,~X_j \mapsto b_j~,
\] is readily seen to be an algebra automorphism. 
In short, the map 
\[
{\rm Aut}(\mathfrak{s}) \to \left\{ b = \sum_{j=2,\ldots,d} c_j X_j~:~c_j \in \mathbb{R},c_2 \not= 0 \right\}~,~\Psi \mapsto \Psi(X_2)~
\] is a bijection. As a consequence, ${\rm Aut}(\mathfrak{s})$ is a $d-1$-dimensional matrix group. 

Hence, returning to the factorization 
\[
A = \lambda \cdot h \cdot A_1~,~\lambda \in \mathbb{R}~,~h \in S~,~A_1  =  \left( \begin{array}{cc} 1 & \mathbf{0} \\ \mathbf{0} & B \end{array} \right), 
\] from Corollary \ref{cor:summary_shearlet}, the possible choices for $B$ can be described as 
\[
B = \left( \begin{array}{ccccc} c_2 & c_3 & \ldots & \ldots & c_d \\
0 & c_2^2 &  \ast & \ast & \ast\\
0 & 0 & c_2^3 & \ast & \ast \\
 0 & 0 & 0 & \ddots & \ast \\
 0 & \ldots & \ldots & \ldots & c_2^{d-1} \end{array} \right) ~,~ (c_2,\ldots,c_d)^T \in \mathbb{R}^* \times \mathbb{R}^{d-1}~.
\] Here the entries above the diagonal and below the first line depend uniquely on $c_2,\ldots,c_d$. They can be determined, either recursively or using the multinomial formula, explicitly from $c_2,\ldots,c_d$; however, we refrain from giving a general formula. To give an idea of the general pattern, we consider the cases $d=3,4,5$. For $d=3,$ one gets the general form
\[
B = \left( \begin{array}{cc} c_2 & c_3  \\ 0 & c_2^2  \end{array} \right)~,
\]
for $d=4$ one has 
\[
B = \left( \begin{array}{ccc} c_2 & c_3 & c_4 \\ 0 & c_2^2 & 2c_2 c_3 \\ 0 & 0 & c_2^3 \end{array} \right) ~,
\] and finally in the case $d=5$ the resulting matrices are of the form 
\[
B = \left( \begin{array}{cccc} c_2 & c_3 & c_4 & c_5\\ 0 & c_2^2 & 2c_2 c_3 & 2 c_2c_4 + c_3^2 \\ 0 & 0 & c_2^3 & c_2^2c_3\\ 0 & 0 & 0 & c_2^4\end{array} \right) ~.
\]

Finally, we need to determine the influence of the scaling subgroup. 
Recall from \cite{AlbertiEtAl2017} that the infinitesimal generators $Y$ have the form $Y_\delta = {\rm diag}(1,1-\delta,1-2\delta,\ldots,1-(d-1)\delta)$. The dimension of $\mathcal{S}_{Co_H}$ depends on $\delta$, as follows: 
\begin{enumerate}
    \item If $\delta=0$, then $\tilde{Y} = I_{d-1}$, and we get 
    \[
    \dim(\mathcal{S}_{Co_H}) = 1+d-1+d-1 = 2d-1~.
    \]
    \item In the case $\delta\not= 0$, all eigenvalues of $\tilde{Y}$ have multiplicity 1, hence the set of possible choices for $B$ consists precisely of the algebra automorphisms of $\mathfrak{s}$ that are diagonal over the canonical basis. This fixes $b_2 = c_2 X_2$, leading to $b_j = c_2^{j-1} X_j$, and as $c_2$ runs through the positive real numbers, one obtains a one-parameter matrix group. Thus we obtain in this case
    \[ \dim(\mathcal{S}_{Co_H}) = 1 + d-1 + 1 = d+1~,\]
    which misses the lower dimension bound from Corollary \ref{cor:SCo_Lie} by 1.
\end{enumerate}




\bibliographystyle{abbrv}
\bibliography{References}
\end{document}